\documentclass[12pt,a4paper]{amsart}
\usepackage{amssymb,latexsym,amsmath,amscd,graphicx,url,color}
\theoremstyle{plain}
\newtheorem{thm}{Theorem}[section]
\newtheorem{prop}[thm]{Proposition}

\newtheorem{lemma}[thm]{Lemma}

\theoremstyle{definition}
\newtheorem{defn}[thm]{Definition}
\newtheorem{rmk}[thm]{Remark}
\newtheorem{rmks}[thm]{Remarks}
\newtheorem{notat}[thm]{Notation}

\newcommand{\G}{{\mathcal G}}
\newcommand{\HH}{{\mathcal H}}
\newcommand{\LL}{{\mathcal L}}
\newcommand{\M}{{\mathcal M}}
\newcommand{\N}{{\mathcal N}}
\newcommand{\NN}{{\mathbb N}}
\newcommand{\U}{{\mathcal U}}
\newcommand{\X}{{\mathcal X}}

\newcommand{\ZZ}{{\mathbb Z}}

\DeclareMathOperator{\hgt}{ht}
\DeclareMathOperator{\lk}{lk}

\title{Gr\"obner bases via linkage}

\author{E. Gorla}
\address{Institut f\"ur Mathematik \\ Universit\"at Basel, 
\hfil\break\indent Rheinsprung 21, CH-4051 Basel, Switzerland} 
\email{elisa.gorla@unibas.ch}
\thanks{The first author was supported by the Swiss National Science
  Foundation under grant no. 123393. She acknowledges financial
  support from the University of Notre Dame and the University of 
  Kentucky, where part of this work was done.}

\author{J. C. Migliore}
\address{Department of Mathematics \\ University of Notre Dame, 
\hfil\break\indent Room 236, Hayes-Healy Building, Notre Dame, IN
46556-5641, USA} 
\email{migliore.1@nd.edu}
\thanks{ The second author was supported
  by the National Security Agency under Grant Number
  H98230-09-1-0031.}

\author{U. Nagel}
\address{Department of Mathematics \\ University of Kentucky,
 \hfil\break\indent 715 Patterson Office Tower, Lexington, KY
 40506-0027, USA}
\email{uwe.nagel@uky.edu}
\thanks{The third author was supported by the National
  Security Agency under Grant Number H98230-09-1-0032.} 


\begin{document}

\begin{abstract} In this paper, we give a sufficient condition for a
  set $\G$ of polynomials to be a Gr\"obner basis with respect to a
  given term-order for the ideal $I$ that it generates. Our criterion
  depends on the linkage pattern of the ideal $I$ and of the ideal
  generated by the initial terms of the elements of $\G$.
  We then apply this criterion to ideals generated by minors and
  pfaffians. More precisely, we consider large families of ideals
  generated by minors or pfaffians in a matrix or a ladder, where the
  size of the minors or pfaffians is allowed to vary in different
  regions of the matrix or the ladder. We use the sufficient condition
  that we established to prove that the minors or pfaffians form a
  reduced Gr\"obner basis for the ideal that they generate, with
  respect to any diagonal or anti-diagonal term-order. We also
  show that the corresponding initial ideal is Cohen-Macaulay and
  squarefree, and that the simplicial complex associated to it is
  vertex decomposable, hence shellable. Our proof relies on known
  results in liaison theory, combined with a simple Hilbert function
  computation. In particular, our arguments are completely algebraic.
\end{abstract}

\maketitle

\section*{Introduction}

Gr\"obner bases are the most widely applicable
computational tool available in the context of commutative
algebra and algebraic geometry. 
However they also are an important theoretical tool, as they can be
used to establish properties such as, e.g., primality,
normality, Cohen-Macaulayness, and to give formulas for the height of
an ideal. Liaison theory, or linkage, on the other hand, is mostly 
regarded as a classification tool. In fact, much effort has been
devoted in recent years to the study of liaison classes, in particular
to deciding which ideals belong to the G-liaison class of a complete
intersection. However, a clear understanding of the liaison pattern of
an ideal often allows us to recursively compute invariants such as its
Hilbert function and graded Betti numbers. 

In this paper we introduce liaison-theoretic methods as a tool in the
theory of Gr\"obner bases. More precisely, we deduce that a certain set
$\G$ of polynomials is a Gr\"obner basis for the ideal $I$ that it
generates by understanding the linkage pattern of the ideal $I$ and of
the monomial ideal generated by the initial terms of the elements of
$\G$. Concretely, we apply this reasoning to ideals generated by minors
or pfaffians, whose liaison pattern we understand.

Ideals generated by minors or pfaffians have been studied extensively
by both commutative algebraists and algebraic geometers. The study of
determinantal rings and varieties is an active area of research per
se, but it has also been instrumental to the development of new
techniques, which have become part of the commonly used tools in
commutative algebra. Ideals generated by minors and pfaffians are
studied in invariant theory and combinatorics, and are relevant in
algebraic geometry. In fact, many classical varieties such as the
Veronese and the Segre varieties are cut out by minors or
pfaffians. Degeneracy loci of morphisms between direct sums 
of line bundles over projective space have a determinantal
description, as do Schubert varieties, Fitting schemes, and some toric
varieties. Ideals generated by minors or pfaffians are
often investigated by commutative algebraists by 
means of Gr\"obner basis techniques (see, e.g., \cite{br88},
\cite{st90}, \cite{he92}). Using such tools, many families of ideals
generated by minors or pfaffians have been shown to enjoy properties
such as primality, normality, and Cohen-Macaulayness.  
A different approach to the study of ideals generated by minors is via
flags and degeneracy loci, and was initiated in~\cite{fu91}. Such an
approach allows one to establish that these ideals are normal,
Cohen-Macaulay and have rational singularities (see, e.g.,
\cite{kn05}, \cite{kn09}, \cite{kn10}).

In recent years, much progress has been made towards understanding
determinantal ideals and varieties also from the point of view of
liaison theory. A central open question in liaison theory asks whether 
every arithmetically Cohen-Macaulay scheme is glicci, i.e., whether it
belongs to the G-liaison class of a complete intersection. In~\cite{ga53},
\cite{kl01}, \cite{go07a}, \cite{go07b}, \cite{go08}, \cite{de09}, 
and~\cite{go10}, several families of ideals generated by minors or 
pfaffians are shown to be glicci. More precisely, it is shown that
they can be obtained from an ideal generated by linear forms via a
sequence of ascending elementary G-biliaisons. Moreover, each of the
elementary G-biliaisons takes place between two ideals which both
belong to the family in question. Since the linkage steps are
described very explicitly, in theory it is possible to use the linkage
results to recursively compute invariants or establish properties of
these ideals. This has been done, e.g., in~\cite{co09} using the
linkage results from~\cite{go07b}.  

Rather than contributing to the theory of liaison (see for
instance~\cite{mi98}), in this paper we give a new method 
of using liaison as a tool.
More precisely, we consider large families of ideals generated by
minors or pfaffians in a matrix or a ladder, namely pfaffian ideals of
ladders, mixed determinantal, and symmetric mixed determinantal
ideals. Combining the liaison results from~\cite{go07b}, \cite{de09},
and~\cite{go10} with a Hilbert function computation, we are able to
prove that the pfaffians or the minors are a reduced Gr\"obner basis
for the ideal that they generate, with respect to any anti-diagonal or
diagonal term-order.  Moreover, we show the simplicial complex
corresponding to the initial ideal of any ideal in the family that we
consider is vertex decomposable. Vertex decomposability is a strong
property, which in particular implies shellability of the complex and
Cohen-Macaulayness of the associated initial ideal. 

In Section~\ref{mainlemma}, we prove a lemma which will be central to
the subsequent arguments (Lemma~\ref{inid}). The lemma gives a
sufficient criterion for a monomial ideal to be the initial ideal of a 
given ideal $J$. Both the ideal $J$ and the ``candidate'' initial
ideal are constructed via Basic Double Linkage or elementary
biliaison. In Section~\ref{ex_sect} we use Lemma~\ref{inid} to prove
that the maximal minors of a matrix of indeterminates are a Gr\"obner
basis of the ideal that they generate with respect to any diagonal (or
anti-diagonal) term-order. Although the result is well-known, we wish
to illustrate our method by showing how it applies to this simple
example. In Sections~\ref{pfaff_sect}, \ref{symm_sect}, and
\ref{ladd_sect}, we apply our technique to ideals generated by: 
pfaffians of mixed size in a ladder of a skew-symmetric matrix of
indeterminates, minors of mixed size in a ladder of a symmetric matrix
of indeterminates, and minors of mixed size in a ladder of a
matrix of indeterminates. We prove that the natural generators of
these ideals are a Gr\"obner basis with respect to any diagonal (in
the case of minors of a symmetric matrix) or anti-diagonal (in the
case of pfaffians or minors in a generic matrix) term-order. We also
prove that the corresponding initial ideals are Cohen-Macaulay, and
that the associated simplicial complexes are vertex decomposable.
While Sections~\ref{mainlemma} and~\ref{ex_sect} are meant to be read
first, Sections~\ref{pfaff_sect}, \ref{symm_sect}, and~\ref{ladd_sect}
can be read independently of each other, and in any order. In the
appendix, we indicate how our liaison-theoretic approach can be made
self-contained in order to derive also all the classical
Gr\"obner basis results about ladder determinantal ideals from
one-sided ladders.

\section{Linkage and Gr\"obner bases}\label{mainlemma}

Let $K$ be an arbitrary field and let $R$ be a standard graded
polynomial ring in finitely many indeterminates over $K$. 
In this section, we give a sufficient condition for a set $\G$
of polynomials to be a Gr\"obner basis with respect to a given
term-order for the ideal $I$ that it generates. Our criterion depends
on the linkage pattern of the ideal $I$ and of the monomial ideal
generated by the initial terms of the elements of $\G$. 

In order to use geometric language, we need to consider the
algebraic closure of the field $K$. Notice however that restricting
the field of coefficients does not affect the property of being a
Gr\"obner basis, as long as the polynomials are defined over the
smaller field. More precisely, if $I=(g_1,\ldots,g_s)\subset
K[x_0,\ldots,x_n]$ and $g_1,\ldots,g_s$ have coefficients in a
subfield $k$ of $K$, then: $g_1,\ldots,g_s$ are a Gr\"obner
basis of $I\subseteq K[x_0,\ldots,x_n]$ if and only if they are a
Gr\"obner basis of $I\cap k[x_0,\ldots,x_n]$.
In this sense, the property of being a Gr\"obner basis
does not depend on the field of definition. Therefore, while proving
that a set $\G$ of polynomials is a  Gr\"obner basis with respect to a
given term-order for the ideal $I$ that it generates, we may pass to
the algebraic closure without loss of generality. We shall therefore
assume without loss of generality that the field $K$ is algebraically
closed. 

\begin{notat} 
Fix a term-order $\sigma$. Let $I\subset R$ be an ideal and let
$\G$ be a set of polynomials in $R$. We denote by $in(I)$ the initial
ideal of $I$ with respect to $\sigma$, and by $in(\G)$ the set of
initial terms of the elements of $\G$ with respect to $\sigma$.
\end{notat}

For the convenience of the reader, we recall the definition of
diagonal and anti-diagonal term-order.

\begin{defn}
Let $X$ be a matrix (resp. a skew-symmetric or a symmetric matrix) of
indeterminates. Let $\sigma$ be a term-order on the set of terms of
$K[X]$. The term-order $\sigma$ is {\bf diagonal} if the leading term
with respect to $\sigma$ of the determinant of a submatrix of $X$ is
the product of the indeterminates on its diagonal. It is {\bf
  anti-diagonal} if the leading term with respect to $\sigma$ of the
determinant of a submatrix of $X$ is the product of the indeterminates
on its anti-diagonal. 
\end{defn}

\begin{notat}
Let $A$ be a finitely generated, graded $R$-module. We denote by
$H_A(d)$ the {\bf Hilbert function} of $A$ in degree $d$, i.e., the
dimension of $A_d$ as a $k$-vector space. 
\end{notat}

In this paper we study large families of ideals generated by minors or
pfaffians in a matrix or a ladder, where the size of the minors or
pfaffians is allowed to vary in different regions of the matrix or the
ladder. We study their initial ideals with respect to a diagonal or
anti-diagonal term-order, and we prove that the associated simplicial
complexes are vertex decomposable. In particular, the initial ideals in
question are Cohen-Macaulay. For the convenience of the reader, we now
recall the main definitions.

\begin{defn} 
A {\bf simplicial complex} $\Delta$ on $n+1$ vertices, is a collection
of subsets of $\{0,\ldots,n\}$ such that for any $F\in\Delta$, if
$G\subseteq F$, then $G\in\Delta$. An $F\in\Delta$ is called a {\bf
  face} of $\Delta$. The dimension of a face $F$ is $\dim F=|F|-1$,
and the {\bf dimension} of the complex is 
$$\dim\Delta=\max\{\dim F\mid F\in\Delta\}.$$
The complex $\Delta=2^{\{0,\ldots,n\}}$ is called a {\bf simplex}.

The {\bf vertices} of $\Delta$ are the subsets of $\{0,\ldots,n\}$ of
cardinality one. The faces of $\Delta$ which are maximal with respect
to inclusion are called {\bf facets}. A complex is {\bf pure} if all
its facets have dimension equal to the dimension of the complex.
\end{defn}

\begin{notat}
To each face $F\in\Delta$ we associate the following two
simplicial subcomplexes of $\Delta$: the {\bf link} of $F$
$$\lk_F(\Delta)=\{G\in\Delta\mid F\cup G\in\Delta, F\cap
G=\emptyset\}$$ and the {\bf deletion}
$$\Delta-F=\{G\in\Delta\mid F\cap G=\emptyset\}.$$
If $F=\{k\}$ is a vertex, we denote the link of $F$ and the deletion by
$\lk_k(\Delta)$ and $\Delta-k$, respectively.
\end{notat}

\begin{defn}
A simplicial complex $\Delta$ is {\bf vertex decomposable} if it is a
simplex, or it is the empty set, or there exists a vertex $k$ such that 
$\lk_k(\Delta)$ and $\Delta-k$ are both pure and vertex decomposable,
and $$\dim\Delta=\dim(\Delta-k)=\dim\lk_k(\Delta)+1.$$
\end{defn}

In this article, we show that the simplicial complexes associated to
the initial ideals of the family of ideals that we consider are vertex
decomposable.

\begin{defn}
The {\bf Stanley-Reisner ideal} associated to a complex $\Delta$ on
$n+1$ vertices is the squarefree monomial ideal
$$I_{\Delta}=(x_{i_1},\ldots,x_{i_s}\mid
\{i_1,\ldots,i_s\}\not\in\Delta)\subset K[x_0,\ldots,x_n].$$
Conversely, to every squarefree monomial ideal $I\subseteq
K[x_0,\ldots,x_n]$ one can associate the unique simplicial complex
$\Delta(I)$ on $n+1$ vertices, such that $I_{\Delta(I)}=I.$
\end{defn}

\begin{rmk}
\begin{enumerate}
\item A vertex $\{k\}\in\Delta$ is called a {\bf cone point} if for every
face $F\in\Delta$, $F\cup\{k\}\in\Delta$. If $\Delta$ has a cone point
$\{k\}$, then $$I_{\Delta}=I_{\Delta-k}K[x_0,\ldots,x_n].$$
Moreover, $\Delta$ is vertex decomposable if and only if $\Delta-k$
is. Therefore, {\em we will not distinguish between a complex and a
  cone over it}.
\item Notice that, if $\Delta$ is a complex on $n+1$ vertices, then
  both $\lk_k(\Delta)$ and $\Delta-k$ are complexes on $n$ (or fewer)
  vertices. However, since we do not distinguish between a complex and
  a cone over it, we will regard them as complexes on $n+1$ vertices.
\item On the side of the associated Stanley-Reisner ideals, let $I$
  and $J$ be squarefree monomial ideals such that the generators of $I$
  involve fewer variables than the generators of $J$. Then we may associate
  to $I$ and $J$ simplicial complexes $\Delta(I)$ and $\Delta(J)$ on
  the same number of variables. This amounts to regarding $I$ and $J$
  as ideals in the same polynomial ring.
\end{enumerate}
\end{rmk}

We now recall some definitions from liaison theory that will be
fundamental throughout the paper.

\begin{defn}\label{g0}
Let $J\subset R$ be a homogeneous, saturated ideal. We say that $J$
is {\bf Gorenstein in codimension $\leq$ c} if the localization
$(R/J)_P$ is a Gorenstein ring for any prime ideal $P$ of $R/J$ of
height smaller than or equal to $c$. We often say that $J$ is
$G_c$. We call {\bf generically Gorenstein}, or $G_0$, an ideal $J$
which is Gorenstein in codimension 0.
\end{defn}

\begin{defn}\label{bdl}
Let $A\subset B\subset R$ be homogeneous ideals such that $\hgt A=\hgt
B- 1$ and $R/A$ is Cohen-Macaulay. Let $f\in R_d$ be a homogeneous element
of degree $d$ such that $A:f=A$. The ideal $C:=A+fB$ is called a {\bf Basic
Double Link} of degree $d$ of $B$ on $A$. If moreover $A$ is $G_0$ and
$B$ is unmixed, then $C$ is a {\bf Basic Double G-Link} of $B$ on $A$.
\end{defn}

\begin{defn}\label{bilid}
Let $I,J\subset R$ be homogeneous, saturated, unmixed ideals, such
that $\hgt(I)=\hgt(J)=c$. We say that $J$ is obtained by an {\bf
elementary biliaison} of height $\ell$ from $I$ if there exists a
Cohen-Macaulay ideal $N$ in $R$ of height $c-1$
such that $N\subseteq I\cap J$ and $J/N\cong [I/N](-\ell)$ as
$R/N$-modules. If in addition the ideal $N$ is $G_0$, then $J$ is
obtained from $I$ via an {\bf elementary G-biliaison}. If $\ell>0$ we
have an {\bf ascending} elementary G-biliaison.
\end{defn}

We refer to~\cite{mi98}, \cite{kl01}, and~\cite{ha07} for the basic 
properties of Basic Double Linkage and elementary biliaison. Notice in
particular that, if $C$ is a Basic Double Link of $B$ on $A$, then it is
not known in general whether $B$ and $C$ belong to the same G-liaison
class. On the other side, if $C$ is a Basic Double G-Link of $B$ on
$A$, then $B$ and $C$ can be G-linked in two steps.

We are now ready to state a sufficient condition for a set
of polynomials to be a Gr\"obner basis (with respect to a given
term-order) for the ideal that they generate.

\begin{lemma}\label{inid}
Let $I,J,N\subset R$ be homogeneous, saturated, unmixed ideals, such that
$N\subseteq I\cap J$ and $\hgt(I)=\hgt(J)=\hgt(N)+1$. Assume that $N$
is Cohen-Macaulay. 
Let $A,B,C\subset R$ be monomial ideals
such that $C\subseteq in(J)$, $A=in(N)$ and $B=in(I)$ with respect to
some term-order $\sigma$. Assume that $A$ is Cohen-Macaulay and
that $\hgt(B)=\hgt(A)+1$. Suppose that $J$ is obtained from $I$ via an
elementary biliaison of height $\ell$ on $N$, and that $C$ is a Basic
Double Link of degree $\ell$ of $B$ on $A$. Then $C=in(J)$.
\end{lemma}

\begin{proof}
Since $C\subseteq in(J)$, it suffices to show that $H_C(d)=H_J(d)$ for
all $d\in\ZZ$. This is indeed the case, since 
$$H_C(d)=H_B(d-\ell)+H_A(d)-H_A(d-\ell)=H_I(d-\ell)+H_N(d)-H_N(d-\ell)=H_J(d).$$
\end{proof}

\begin{rmks}
\begin{enumerate}
\item Notice that, if $J$ is obtained from $I$ via an elementary
  biliaison on $N$, we do not know in general whether they belong 
  to the same G-liaison class. However, if in addition $N$ is
  generically Gorenstein, then $J$ is obtained from $I$ via an
  elementary G-biliaison on $N$. In particular, it can be obtained
  from $I$ via two Gorenstein links on $N$.
\item If in addition $A$ is generically Gorenstein, then $C$ is a
  Basic Double G-Link of $B$ on $A$. In particular, it can be obtained
  from $B$ via two Gorenstein links on $A$. 
\item The concepts of Basic Double Linkage and biliaison are
  interchangeable in the statement of Lemma~\ref{inid}. More
  precisely, Basic Double Linkage is a special case of
  biliaison. Moreover, it can be shown that if $J$ is obtained from
  $I$ via an elementary biliaison of height $\ell$ on $N$, then there
  exist an ideal $H$ and a $d\in\ZZ$ s.t. $H$ is a Basic Double Link
  of degree $d+\ell$ of $I$ on $N$ and also a Basic Double Link of
  degree $d$ of $J$ on $N$. Then it is easy to verify that the lemma
  holds under the weaker assumption that $C$ is obtained from $B$ via
  an elementary biliaison of height $\ell$ on $A$.
\end{enumerate}
\end{rmks}

In the next section, we use the lemma to prove that the maximal minors
of a matrix of indeterminates are a Gr\"obner basis of the ideal that
they generate with respect to any diagonal term-order. Although the
result is well-known, we wish to illustrate our method by showing how it
applies to this simple example.
In Sections~\ref{pfaff_sect}, \ref{symm_sect}, and \ref{ladd_sect}, we
apply the lemma to ideals generated by: 
pfaffians of mixed size in a ladder of a skew-symmetric matrix of
indeterminates, minors of mixed size in a ladder of a symmetric matrix
of indeterminates, and minors of mixed size in a one-sided ladder of a
matrix of indeterminates. We prove that the natural generators of
these ideals are a Gr\"obner basis with respect to any diagonal (in
the case of minors of a symmetric matrix) or anti-diagonal (in the
case of pfaffians or minors in a generic matrix) term-order. We also
prove that their initial ideals are squarefree and that they can be
obtained from an ideal generated by indeterminates via a sequence of
Basic Double G-links of degree 1, which only involve squarefree
monomial ideals. In particular, they are glicci (i.e., they can be
obtained from a complete intersections via a sequence of G-links),
hence they are Cohen-Macaulay. Moreover, we prove that the simplicial
complexes associated to their initial ideals are vertex decomposable. 

Notice that, if we knew a priori that the simplicial complexes
associated to the initial ideals are vertex decomposable, then we
could deduce that the corresponding squarefree monomial ideals are
glicci by the following result of Nagel and R\"omer. However, we
cannot directly apply their result in our situation, since we need to
first produce the Basic Double G-links on the squarefree monomial
ideals, in order to deduce that the associated simplicial complexes
are vertex decomposable.

\begin{thm}[\cite{na08}, Theorem~3.3]
Let $\Delta$ be a simplicial complex on $n+1$ vertices and let
$I_{\Delta}\subset K[x_0,\ldots,x_n]$ be the Stanley-Reisner ideal of
$\Delta$. Assume that $\Delta$ is (weakly) vertex decomposable. Then
$I_{\Delta}$ can be obtained from an ideal generated by indeterminates
via a sequence of Basic Double G-links of degree 1, which only involve
squarefree monomial ideals.
\end{thm}

Notice that, although the statement above is slightly stronger than
Theorem~3.3 in~\cite{na08}, the result above follows from the proof
in~\cite{na08}.

\section{A simple example: ideals of maximal minors}\label{ex_sect}

This section is meant to illustrate the idea and the method of our
proof on a simple example. We prove that the maximal minors
of a matrix of indeterminates are a Gr\"obner basis of the ideal that
they generate with respect to any diagonal term-order. Notice that for
the case of minors in a matrix, diagonal term-orders are the same as
anti-diagonal ones, up to transposing the matrix. 

\begin{thm}\label{maxmin}
Let $X=(x_{ij})$ be an $m\times n$ matrix whose entries are distinct
indeterminates, $m\leq n$. Let $K[X]=K[x_{ij} \;|\; 1\leq i,j\leq n ]$
be the polynomial ring associated to $X$. Let $\G_m(X)$ be the set of
maximal minors of $X$ and let $I_m(X)\subset K[X]$ be the ideal
generated by $\G_m(X)$. Let $\sigma$ be a diagonal
term-order and let $in(I_m(X))$ be the initial ideal of $I_m(X)$
with respect to $\sigma$. Then $\G_m(X)$ is a reduced Gr\"obner basis
of $I_m(X)$ with respect to $\sigma$ and $in(I_m(X))$ is a squarefree,
Cohen-Macaulay ideal. Moreover, the simplicial complex $\Delta_X$
associated to $in(I_m(X))$ is vertex decomposable.
\end{thm}

\begin{proof}
We proceed by induction on $mn=|X|$. If $|X|=1$, then the ideal 
$I_1(X)$ is generated by one indeterminate. Hence
$\G_1(X)$ is a reduced Gr\"obner basis of $I_1(X)$ with
respect to any term ordering and $I_1(X)=in(I_1(X))$ is generated by
indeterminates. The associated simplicial complex $\Delta_X$ is the
empty set, hence it is vertex decomposable.

By induction hypothesis, in order to prove the thesis for a matrix
with $m$ rows and $n$ columns, we may assume that it holds for any
matrix with fewer than $mn$ entries. If $m=1$, then $\G_1(X)$ consists
of indeterminates, hence it is a reduced Gr\"obner basis of $I_1(X)$ with
respect to any term ordering. Moreover, $I_1(X)=in(I_1(X))$ is
generated by indeterminates. The associated simplicial complex
$\Delta_X$ is the empty set, hence it is vertex decomposable.

If $m\geq 2$, let $C\subseteq in(I_m(X))$ 
be the ideal generated by the initial terms of $\G_m(X)$. We claim
that $C=in(I_m(X))$. In fact, let $Z$ be the $m\times(n-1)$ matrix
obtained from $X$ by deleting the last column, and let $Y$ be the
$(m-1)\times(n-1)$ matrix obtained from $Z$ by deleting the last
row. Let $A=in(\G_m(Z))$ be the ideal generated by the initial terms
of the elements of $\G_m(Z)$, and let $B=in(\G_{m-1}(Y))$ be the ideal
generated by the initial terms of the elements of
$\G_{m-1}(Y)$. By the induction hypothesis, $\G_m(Z)$ is a Gr\"obner
basis for $I_m(Z)$ and $\G_{m-1}(Y)$ is a Gr\"obner basis for
$I_{m-1}(Y)$. In other words, $A=in(I_m(Z))$ and $B=in(I_{m-1}(Y))$. Notice that
$$in(\G_m(X))=in(\G_m(Z))\cup x_{mn}in(\G_{m-1}(Y))$$ where $x_{mn}\G$
denotes the set of products $x_{mn}g$ for $g\in\G$. Since
$x_{mn}$ does not appear in $in(\G_m(Z))$, we have $A:x_{mn}=A$. Therefore, 
$$A+x_{mn}B=C\subseteq in(I_m(X))$$ and $C$ is a
Basic Double G-Link of degree 1 of $B$ on $A$. $A$ and $B$ are
squarefree and glicci by induction hypothesis, therefore $C$ is
squarefree and glicci. 
It follows from~\cite[Theorem 3.6]{kl01} that $I_m(X)$ is obtained from
$I_{m-1}(Y)$ via an elementary G-biliaison of height 1 on
$I_m(Z)$. By Lemma~\ref{inid} the maximal minors of $X$ are a
Gr\"obner basis of $I_m(X)$ with respect to $\sigma$, and
$C=in(I_m(X))$. 

Finally, let $\Delta_Z,\Delta_Y,\Delta_X$ be the
simplicial complexes associated to $A,B,C$, respectively. Since
$A+x_{mn}B=C$, $$\Delta_Z=\Delta_X-mn\;\;\;\mbox{and}\;\;\;
\Delta_Y=\lk_{mn}(\Delta_X).$$ 
Since $\Delta_Y$ and $\Delta_Z$ are vertex decomposable by induction
hypothesis, so is $\Delta_X$.
\end{proof}

\begin{rmk}
Theorem~\ref{maxmin} gives in particular a new proof of the fact that
the maximal minors of a generic matrix are a Gr\"obner basis for the
ideal that they generate, with respect to a diagonal or anti-diagonal
term-order. This is a classical result. While previous proofs have a
combinatorial flavor, our proof is completely algebraic, and
independent of all the previous Gr\"obner basis results. 
\end{rmk}

\section{Pfaffian ideals of  ladders}\label{pfaff_sect}

In this section, we study Gr\"obner bases with respect to an
anti-diagonal term-order of ideals generated by pfaffians. We always 
consider pfaffians in a skew-symmetric matrix whose entries are
distinct indeterminates. 

Pfaffians of size $2t$ in a skew-symmetric matrix are known to be a
Gr\"obner basis for the ideal that they generate, as shown by Herzog
and Trung in~\cite{he92} and independently by Kurano in~\cite{ku91}. 
In~\cite{de98}, De Negri generalized this
result to pfaffians of size $2t$ in a symmetric ladder. In this
section, we extend these results to pfaffians of mixed size in a
symmetric ladder. In other words, we consider ideals generated by
pfaffians, whose size is allowed to vary in different regions of the
ladder (see Definition~\ref{idealpf}). In Theorem~\ref{gbpf} we prove
that the pfaffians are a reduced Gr\"obner basis with respect to any
anti-diagonal term-order for the ideal that they generate, and that the
corresponding initial ideal is Cohen-Macaulay and
squarefree. Moreover, the associated simplicial complex is vertex
decomposable. The proof that we give is not a generalization of the
earlier ones. Instead, we use our liaison-theoretic approach and the
linkage results of~\cite{de09}.

In the recent paper~\cite{de09u}, De Negri and Sbarra consider a
different family of ideals generated by pfaffians of mixed size in a
skew-symmetric matrix, namely cogenerated ideals. They are able to
show that the pfaffians are almost never a Gr\"obner basis of the
ideal that they generate with respect to an anti-diagonal
term-order. The family of ideals that they study and the family that we 
consider in this article have a small overlap, which consists of ideals
of pfaffians of size $2t$ in a symmetric ladder, and of ideals
generated by $2t$-pfaffians in the first $m$ rows and columns of the
matrix and $(2t+2)$-pfaffians in the whole matrix. For the ideals in the
overlap, the pfaffians are a Gr\"obner basis for the ideal that they
generate. This follows from Theorem~2.8 of~\cite{de09u}, as well as
from our Theorem~\ref{gbpf}. The results in~\cite{de09u} and those in
this article are obtained independently and by following a completely 
different approach. Nevertheless, we feel that they complement each
other nicely, giving a more complete picture of the behavior of
Gr\"obner bases of pfaffian ideals and of their intrinsic complexity.

Pfaffian ideals of ladders were introduced and studied by De Negri and
the first author in~\cite{de09}. From the point of view of liaison
theory, this is a very natural family to consider. 
In this section, we prove that pfaffians of
mixed size in a ladder of a skew-symmetric matrix are a Gr\"obner
basis with respect to any anti-diagonal term-order for the ideal that
they generate. We start by introducing the relevant definitions and
notation.

Let $X=(x_{ij})$ be an $n\times n$ skew-symmetric matrix of indeterminates.
In other words, the entries
$x_{ij}$ with $i<j$ are indeterminates, $x_{ij}=-x_{ji}$ for
$i>j$, and $x_{ii}=0$ for all $i=1,...,n$.
Let $K[X]=K[x_{ij} \;|\; 1\leq i<j\leq n ]$ be the polynomial ring associated 
to $X$. 

\begin{defn}\label{laddpf}
A {\bf symmetric ladder} $\mathcal L$ of $X$ is a subset of the set
$\X=\{(i,j)\in\NN^2 \;|\; 1\le i,j\le n\}$
with the following properties :
\begin{enumerate}
\item if $(i,j)\in {\mathcal L}$ then $(j,i)\in {\mathcal L}$,
\item if  $i<h,j>k$ and $(i,j),(h,k)\in\mathcal L$, then also
$(i,k),(i,h),(h,j),(j,k)\in\mathcal L$.
\end{enumerate}

We do not assume that the ladder $\LL$ is connected, nor that
$X$ is the smallest skew-symmetric matrix having $\LL$ as ladder. 
It is easy to see that any symmetric ladder can be decomposed as a
union of square subladders 
\begin{equation}\label{decomppf}
\LL=\X_1\cup\ldots\cup \X_s
\end{equation} 
where
 $$\X_k=\{(i,j)\;|\; a_k\le i,j \le b_k\},$$ 
for some  integers $1\leq a_1\leq\ldots\leq a_s\leq n$ and $1\leq
b_1\leq\ldots\leq b_s\leq n$ such that $a_k<b_k$ for all $k$. 
We say that $\LL$ is the ladder with {\bf upper
corners} $(a_1,b_1),\ldots,(a_s,b_s)$, and that $\X_k$ is the
square subladder of $\LL$ with upper outside corner $(a_k,b_k)$. See
Figure~\ref{decomp_fig}. 
\begin{figure}[h!]
\input{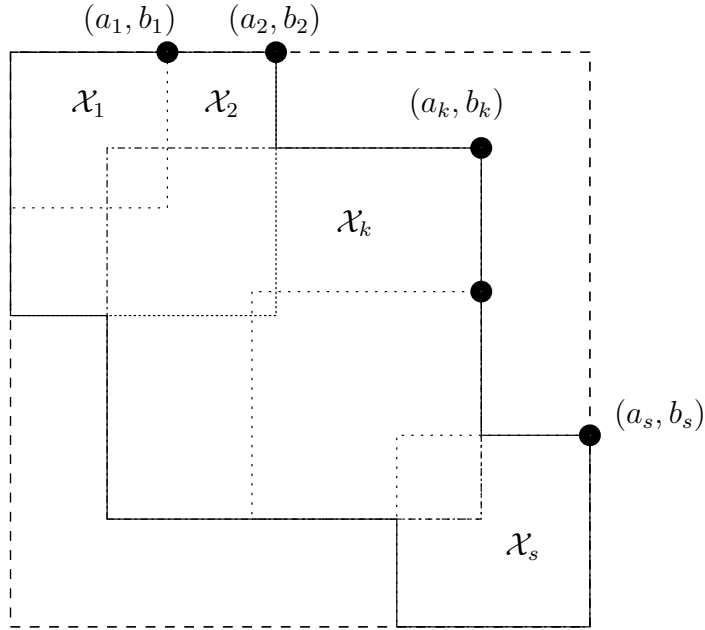}
\caption{An example of a symmetric ladder with its decomposition as a
  union of skew-symmetric matrices and the corresponding upper
  corners.}
\label{decomp_fig}
\end{figure}
We allow two upper corners to have the
same first or second coordinate, however we assume that no two
upper corners coincide. We assume moreover that all upper corners
belong to the border of the ladder, i.e.,
$(a_k-1,b_k+1)\not\in\LL$. Notice that with these conventions a 
ladder does not have a unique decomposition of the form
(\ref{decomppf}). In other words, a symmetric ladder does not
correspond uniquely to a set of upper corners
$(a_1,b_1),\ldots,(a_s,b_s)$. However, any symmetric ladder is 
determined by its upper corners as in (\ref{decomppf}). 
Moreover, the upper corners of $\mathcal L$ determine the
submatrices $\X_k$. We assume that every symmetric ladder comes with its
set of upper corners and the corresponding decomposition as a
union of square submatrices as in (\ref{decomppf}). 
Notice that the set of upper corners as given in our definition
contains all the usual upper outside corners, and may contain some of
the usual upper inside corners, as well as other elements of the
ladder which are not corners of the ladder in the usual sense. 

Given a ladder $\mathcal L$ we set $L=\{x_{ij}\in X\;|\; (i,j)\in
{\mathcal L},\; i<j\}$. If $p$ is a positive integer, we let  
$I_{2p}(L)$ denote the ideal generated by the set of the
$2p$-pfaffians of $X$ which involve only indeterminates of
$L$. In particular $I_{2p}(X)$ is the ideal of $K[X]$ generated by 
the $2p$-pfaffians of $X$.
\end{defn}

\begin{defn}\label{idealpf}
Let $\LL=\X_1\cup\ldots\cup \X_s$ be a symmetric ladder.
Let $X_k=\{x_{i,j}\;|\; (i,j)\in\X_k,\; i<j\}$ for $k=1,\dots,s$.
Fix a vector $t=(t_1,\ldots,t_s)$, $t\in\ZZ_+^s$.
The {\bf ladder pfaffian ideal} $I_{2t}(L)$ is by definition the sum of 
pfaffian ideals $I_{2t_1}(X_1)+\ldots+I_{2t_s}(X_s)$. We also
refer to these ideals as {\bf pfaffian ideals of ladders}. For ease of
notation, we regard all ladder pfaffian ideals as ideals in $K[X]$.
\end{defn}

This family of ideals was introduced and studied in~\cite{de09}. From
the point of view of G-biliaison, this appears to be the right family
to consider. Notice that it does not coincide with the family
of cogenerated pfaffian ideals as defined, e.g., in~\cite{de95}. 

\begin{notat}\label{genspf}
Denote by $\G_{2t_k}(X_k)$ the set of the $2t_k$-pfaffians of $X$
which involve only indeterminates of $X_k$ and
let $$\G_{2t}(L)=\G_{2t_1}(X_1)\cup\ldots\cup\G_{2t_s}(X_s).$$ 
The elements of $\G_{2t}(L)$ are a minimal system of generators of
$I_{2t}(L)$. We sometimes refer to them as ``natural generators''.
\end{notat}

\begin{notat}\label{ladderheight}
For a symmetric ladder $\LL$ with upper corners
$(a_1,b_1),\ldots,(a_s,b_s)$ and $t=(t_1,\ldots,t_s)$, we denote by
$\tilde{\LL}$ the symmetric ladder with upper corners
$(a_1+t_1-1,b_1-t_1+1),\ldots,(a_s+t_s-1,b_s-t_s+1)$. See
Figure~\ref{ladd_hgt_fig}. 
\begin{figure}[h!]
\input{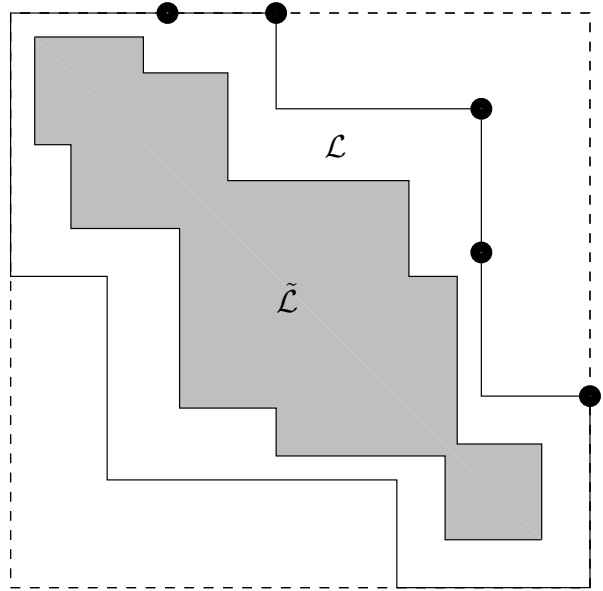}
\caption{An example of a ladder $\LL$ with five upper corners and
  $t=(2,3,4,2,3)$. The corresponding $\tilde{\LL}$ is shaded.}
\label{ladd_hgt_fig}
\end{figure}
\end{notat}

The ladder $\tilde{\LL}$ computes the height of the ideal $I_{2t}(L)$
as follows.

\begin{prop}[Proposition~1.10, \cite{de09}]
Let $\LL$ be the symmetric ladder with  upper corners
$(a_1,b_1), \ldots,$ $(a_s,b_s)$ and $t=(t_1,\ldots,t_s)$. Let 
$\tilde{\LL}$ be as in Notation~\ref{ladderheight}. Then $\tilde{\LL}$
is a symmetric ladder and the height of $I_{2t}(L)$ is equal to the
cardinality of $\{(i,j)\in\tilde{\LL} \;|\; i<j\}$.
\end{prop}

The following is the main result of~\cite{de09}. Its proof consists of
an explicit description of the G-biliaison steps, which will be used in the
proof of Theorem~\ref{gbpf}. 

\begin{thm}[Theorem~2.3, \cite{de09}]\label{biliaisonpf}
Any pfaffian ideal of ladders can be obtained from an ideal generated
by indeterminates by a finite sequence of ascending elementary
G-biliaisons. 
\end{thm}

By combining Lemma~\ref{inid} and Theorem~\ref{biliaisonpf}, we
prove that the pfaffians are a Gr\"obner basis of the ideal that they
generate with respect to any anti-diagonal term-order. 

\begin{thm}\label{gbpf}
Let $X=(x_{ij})$ be an $n\times n$ skew-symmetric matrix of
indeterminates. Let $\LL=\X_1\cup\ldots\cup\X_s$ and
$t=(t_1,\ldots,t_s)$. Let $I_{2t}(L)\subset K[X]$ be the
corresponding ladder pfaffian ideal and let $\G_{2t}(L)$ be the set of
pfaffians that generate it. Let $\sigma$ be any anti-diagonal
term-order. Then $\G_{2t}(L)$ is a reduced Gr\"obner 
basis of $I_{2t}(L)$ with respect to $\sigma$. Moreover, the initial
ideal of $I_{2t}(L)$ with respect to $\sigma$ is squarefree, and the
associated simplicial complex is vertex decomposable. In particular,
the initial ideal of $I_{2t}(L)$ is Cohen-Macaulay.
\end{thm}

\begin{proof}
Let $$I_{2t}(L)=I_{2t_1}(X_1)+\cdots+I_{2t_s}(X_s)\subset K[X]$$ be
the pfaffian ideal of the ladder $\LL$ with $t=(t_1,\ldots,t_s)$ and
upper corners $(a_1,b_1),\ldots,$ $(a_s,b_s)$. Let $\G_{2t}(L)$ be the
set of pfaffians that generate $I_{2t}(L)$. We proceed by induction on
$\ell=|\LL|$.

If $\ell=|\LL|=1$, then $\LL=\tilde{\LL}$ and $t=1$. $\G_{2t}(L)$
consists only of one indeterminate, in particular it is a reduced
Gr\"obner basis with respect to any term-order $\sigma$ of the ideal
that it generates. Since $in(I_2(L))=I_2(L)$ is generated by
indeterminates, it is squarefree and Cohen-Macaulay, and the
associated simplicial complex is the empty set.

We assume that the thesis holds for ideals associated to ladders $\N$
with $|\N|<\ell$ and we prove it for an ideal $I_{2t}(L)$ associated
to a ladder $\LL$ with $|\LL|=\ell$. If
$t_1=\ldots=t_s=1$, then $\G_{2t}(L)$ consists only of
indeterminates. In particular, it is a reduced Gr\"obner basis of the
ideal that it generates, with respect to any term-order
$\sigma$. Moreover, $in(I_2(L))=I_2(L)$ is generated by
indeterminates, hence it is squarefree and Cohen-Macaulay. The
associated simplicial complex is the empty set. Otherwise, let
$k\in\{1,\ldots,s\}$ such that 
$t_k=\max\{t_1,\ldots,t_s\}\geq 2$. Let $\mathcal L'$ be the ladder
with upper corners $$(a_1,b_1),\ldots,(a_{k-1},b_{k-1}),(a_k+1,b_k-1), 
(a_{k+1},b_{k+1}),\ldots,(a_s,b_s)$$
and let $t'=(t_1,\ldots,t_{k-1},t_k-1,t_{k+1},\ldots,t_s)$. 
Let $I_{2t'}(L')\subset K[X]$ be the associated ladder pfaffian
ideal. Let $\G_{2t'}(L')$ be the set of pfaffians which minimally
generate $I_{2t'}(L')$. Since $|\LL'|<\ell$, by
induction hypothesis $\G_{2t'}(L')$ is a reduced Gr\"obner
basis of $I_{2t'}(L')$ with respect to any anti-diagonal
term-order. Hence $$in(I_{2t'}(L'))=(in(\G_{2t'}(L'))).$$
Let $\mathcal M$ be the ladder obtained from $\LL$ by removing 
$(a_k,b_k)$ and $(b_k,a_k)$. ${\mathcal M}$ has upper corners
$$(a_1,b_1),\ldots,(a_{k-1},b_{k-1}),(a_k,b_k-1),
(a_k+1,b_k),(a_{k+1},b_{k+1}),\ldots,(a_s,b_s)$$ and
$u=(t_1,\dots,t_{k-1},t_k,t_k,t_{k+1},\dots,t_s)$. 
Let $I_{2u}(M)\subset K[X]$ be the associated ladder pfaffian
ideal. Let $\G_{2u}(M)$ be the set of pfaffians which minimally
generate $I_{2u}(M)$. Since $|\M|=\ell-1$, by induction hypothesis
$\G_{2u}(M)$ is a reduced Gr\"obner basis of $I_{2u}(M)$ with respect
to any anti-diagonal term-order,
and $$in(I_{2u}(M))=(in(\G_{2u}(M))).$$

It follows from~\cite{de09}, Theorem~2.3 that $I_{2t}(L)$ is
obtained from $I_{2t'}(L')$ via an ascending elementary G-biliaison of
height $1$ on $I_{2u}(M)$. The ideals $in(I_{2t'}(L'))$ and
$in(I_{2u}(M))$ are Cohen-Macaulay by induction hypothesis. Moreover
$$in(\G_{2t}(L))=in(\G_{2u}(M))\cup x_{a_k,b_k}in(\G_{2t'}(L')),$$
where $x_{a,b}\G$ denotes the set of products $x_{a,b}g$ for $g\in\G$. Since
$x_{a_k,b_k}$ does not appear in $in(\G_{2u}(M))$, it does not divide zero
modulo the ideal $in(I_{2u}(M))$. Therefore, 
\begin{equation}\label{cpx_pfaff}
I:=(in(\G_{2t}(L)))=in(I_{2u}(M))+x_{a_k,b_k}in(I_{2t'}(L'))
\subseteq in(I_{2t}(L))
\end{equation}
and $I$ is a Basic Double G-Link of degree 1 of $in(I_{2t'}(L'))$ on
$in(I_{2u}(M))$. Therefore $I$ is a squarefree Cohen-Macaulay
ideal. By Lemma~\ref{inid} $I=in(I_{2t}(L))$, hence $\G_{2t}(L)$ is a
Gr\"obner basis of $I_{2t}(L)$ with respect to any anti-diagonal
term-order.

Let $\Delta$ be the simplicial complex associated to
$in(I_{2t}(L))$. By (\ref{cpx_pfaff}) the simplicial complexes
associated to $in(I_{2t'}(L')$ and $in(I_{2u}(M))$ are
$\lk_{(a_k,b_k)}(\Delta)$ and $\Delta-(a_k,b_k)$,
respectively. $\Delta$ is vertex decomposable, since
$\lk_{(a_k,b_k)}(\Delta)$ and $\Delta-(a_k,b_k)$ are by induction
hypothesis. 
\end{proof}

\begin{rmks}\label{blahblah}
\begin{enumerate}
\item From the proof of the theorem it also follows that $I_{2t}(L)$
  is obtained from an ideal generated by indeterminates via a sequence
  of degree 1 Basic Double G-links, which only involve squarefree
  monomial ideals. Hence in particular it is glicci. 
  Since any vertex decomposable complex is 
  shellable, it also follows that the associated simplicial
  complex is shellable (see Section~5 of~\cite{na08} for a summary 
  of the implications among different properties of simplicial
  complexes, such as vertex decomposability, shellability,
  Cohen-Macaulayness, etc).
\item The proof of Theorem~\ref{gbpf} given above does not
constitute a new proof of the fact that the $2t$-pfaffians in a matrix
or in a symmetric ladder are a Gr\"obner basis with 
respect to any anti-diagonal term-order for the ideal that they
generate. In fact, our proof is based on Theorem~2.3 in~\cite{de09},
which in turn relies on the fact that pfaffians all of the same size
in a ladder of a skew-symmetric matrix generate a prime ideal. Primality
of the ideal is classically deduced from the fact that the pfaffians
are a Gr\"obner basis. So we are extending (and not
reproving) the results in~\cite{he92}, \cite{ku91}, and~\cite{de98}.
\end{enumerate}
\end{rmks}

\section{Symmetric mixed ladder determinantal ideals}\label{symm_sect}

In this section, we study ideals generated by minors contained in a
ladder of a generic symmetric matrix. We show that the minors are 
Gr\"obner bases for the ideals that they generate, with respect to a
diagonal term-order. We also show that the corresponding initial ideal
is glicci (hence Cohen-Macaulay) and squarefree, and that the
associated simplicial complex is vertex decomposable. 

Cogenerated ideals of minors in a symmetric matrix of indeterminates
or a symmetric ladder thereof were studied by Conca in~\cite{co94c}
and~\cite{co94a}. We refer to~\cite{co94c} and~\cite{co94a} for the
definition of cogenerated determinantal ideals in a symmetric
matrix. In those articles Conca proved among other things that the
natural generators of cogenerated ideals of ladders of a symmetric
matrix are a Gr\"obner bases with respect to any diagonal term-order. 
In this section, we study the family of symmetric mixed ladder
determinantal ideals. This family strictly contains the family of
cogenerated ideals. Symmetric mixed ladder determinantal ideals have been
introduced and studied by the first author in~\cite{go10}. This is a
very natural family to study, from the point of view of liaison theory. 
In this paper we extend the result of Conca and prove that the natural
generators of symmetric mixed ladder determinantal ideals are a
Gr\"obner bases with respect to any diagonal term-order. 
 
Let $X=(x_{ij})$ be an $n\times n$ symmetric
matrix of indeterminates. In other words, the entries $x_{ij}$ with
$i\leq j$ are distinct indeterminates, and $x_{ij}=x_{ji}$ for
$i>j$. Let $K[X]=K[x_{ij}\mid 1\leq i\leq j\leq n]$ be the
polynomial ring associated to the matrix $X$. In the sequel, we study
ideals generated by the minors contained in a ladder of a generic
symmetric matrix. Throughout the section, we let 
$$\X=\{(i,j)\mid 1\leq i,j\leq n\}.$$ 
We let $\LL$ be a symmetric ladder (see Definition~\ref{laddpf}). We
can restrict ourselves to 
symmetric ladders without loss of generality, since the ideal
generated by the minors in a ladder of a symmetric matrix coincides
with the ideal generated by the minors in the smallest symmetric
ladder containing it. We do not assume that $\LL$ is connected, nor that
that $X$ is the smallest symmetric matrix having $\LL$ as a ladder. 
Let $$\X^+=\{(i,j)\in\X\mid 1\leq i\leq j\leq n\}\;\;\;\mbox{and}\;\;\; 
\LL^+=\LL\cap\X^+.$$
Since $\LL$ is symmetric, $\LL^+$ determines $\LL$ and vice versa. We
will abuse terminology and call $\LL^+$ a ladder.
Observe that $\LL^+$ can be written as
$$\LL^+=\{(i,j)\in\X^+\mid i\leq c_l \mbox{ or } j\leq d_l \mbox{ for } 
l=1,\ldots,r\; \mbox{ and }$$ $$ i\geq a_l \mbox{ or } j\geq b_l 
\mbox{ for } l=1,\ldots,u \}$$ 
for some  integers $1\leq a_1<\ldots<a_u\leq n$, $n\geq b_1>\ldots>b_u\geq 1$, 
$1\leq c_1<\ldots<c_r\leq n$, and $n\geq d_1>\ldots>d_r\geq 1$, with
$a_l\leq b_l$ for $l=1,\ldots,u$ and $c_l\leq d_l$ for $l=1,\ldots,r$.

The points $(a_1,b_2),\ldots,(a_{u-1},b_u)$ are the {\bf upper outside
corners} of the ladder, $(a_1,b_1),\ldots,(a_u,b_u)$ are the
{\bf upper inside corners}, $(c_2,d_1),\ldots,(c_r,d_{r-1})$ the {\bf
lower outside corners}, and $(c_1,d_1),\ldots,(c_r,d_r)$ the {\bf lower inside
corners}. If $a_u\neq b_u$, then $(a_u,a_u)$ is an upper outside corner
and we set $b_{u+1}=a_u$. Similarly, if $c_r\neq d_r$ then $(d_r,d_r)$
is a lower outside corner, and we set $c_{r+1}=d_r$. A ladder has at least
one upper and one lower outside 
corner. Moreover, $(a_1,b_1)=(c_1,d_1)$ is both an upper and a lower
inside corner. See Figure~\ref{corners}. 
\begin{figure}[h!]
\begin{center}
\input{figs/symm1.pstex_t}
\caption{An example of ladder with tagged lower and upper corners.}
\label{corners}
\end{center}
\end{figure}
The {\bf lower border} of $\LL^+$
consists of the elements $(c,d)$ of $\LL^+$ such that either $c_l\leq
c\leq c_{l+1}$ and $d=d_l$, or $c=c_l$ and $d_l\leq d\leq d_{l-1}$ for
some $l$. See Figure~\ref{lowerborder}. 
\begin{figure}[h!]
\begin{center}
\input{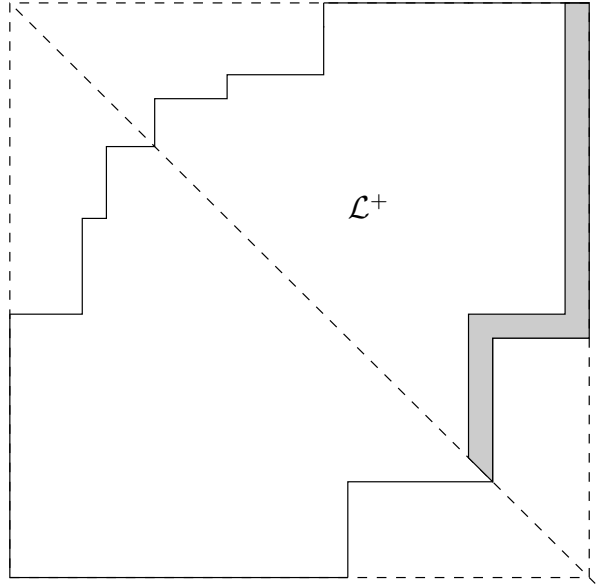}
\caption{The lower border of the same ladder.}
\label{lowerborder}
\end{center}
\end{figure}
All the corners belong to $\LL^+$. In fact, the ladder $\LL^+$
corresponds to its set of lower and upper outside (or equivalently
lower and upper inside) corners. The lower corners
of a ladder belong to its lower border.

Given a ladder $\LL$ we set $L=\{x_{ij}\in X\mid (i,j)\in\LL^+\}$. For
$t$ a positive integer
we let $I_t(L)$ denote the ideal generated by the set of the
$t$-minors of $X$ which involve only indeterminates of
$L$. In particular $I_t(X)$ is the ideal of $K[X]$ generated by 
the minors of $X$ of size $t\times t$.

\begin{notat}\label{decompsym}
Let $\LL^+$ be a ladder. For $(v,w)\in\LL^+$ let 
$$\LL^+_{(v,w)}=\{(i,j)\in\LL^+\mid i\leq v,\; j\leq w\},
\;\;\;\;\; L_{(v,w)}=\{x_{ij}\in X\mid (i,j)\in\LL^+_{(v,w)}\}.$$
Notice that $\LL^+_{(v,w)}$ is a ladder and 
$$\LL^+=\bigcup_{(v,w)\in\U}\LL^+_{(v,w)}$$
where $\U$ denotes the set of lower outside corners of $\LL^+$.
\begin{figure}[h!]
\begin{center}
\input{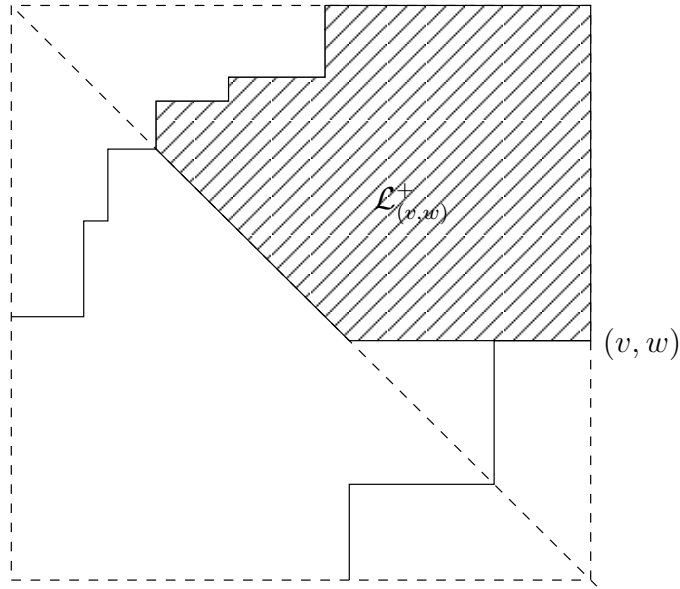}
\caption{The ladder $\LL^+$ with a shaded subladder $\LL^+_{(v,w)}$.}
\label{region}
\end{center}
\end{figure}
\end{notat}

\begin{defn}\label{ldi}
Let $\{(v_1,w_1),\ldots,(v_s,w_s)\}$ be a subset of the
lower border of $\LL^+$ which contains all the lower outside
corners. We order them so that $1\leq v_1\leq\ldots\leq v_s\leq n$
and $n\geq w_1\geq\ldots\geq w_s\geq 1$. Let
$t=(t_1,\ldots,t_s)\in\ZZ_+^s$. Denote $\LL^+_{(v_k,w_k)}$ by 
$\LL^+_k$, and $L_{(v_k,w_k)}$ by $L_k$. 
The ideal $$I_t(L)=I_{t_1}(L_1)+\ldots+I_{t_s}(L_s)\subset
K[X]$$ is a {\bf symmetric mixed ladder determinantal ideal}. 
Denote $I_{(t,\ldots,t)}(L)$ by $I_t(L)$.
We call $(v_1,w_1),\ldots,(v_s,w_s)$ {\bf distinguished points} of
$\LL^+$.
\end{defn}
If $t=(t,\ldots,t)$, then $I_t(L)$ is the ideal generated by the
$t$-minors of $X$ that involve only indeterminates from $L$. These
ideals have been classically studied (see, e.g., \cite{co94c}, \cite{co94a},
\cite{co94b}). It is not hard to show
(see~\cite{go10}, Examples~1.5) that the family of symmetric mixed
ladder determinantal ideals contains the family of cogenerated ideals
in a ladder of a symmetric matrix, as defined in~\cite{co94a}. 

\begin{notat}\label{genssym}
Denote by $\G_{t_k}(L_k)$ the set of the $t_k$-minors of $X$
which involve only indeterminates of $L_k$ and
let $$\G_t(L)=\G_{t_1}(L_1)\cup\ldots\cup\G_{t_s}(L_s).$$ 
The elements of $\G_t(L)$ are a minimal system of generators of
$I_t(L)$. We sometimes refer to them as ``natural generators''.
\end{notat}

\begin{notat}
Let $\LL$ be a ladder with distinguished points
$(v_1,w_1),\ldots,(v_s,w_s)$. We denote
by $$\tilde{\LL}^+=\{(i,j)\in\LL^+\mid i\leq v_{k-1}-t_{k-1}+1 \mbox{
  or } j\leq w_k-t_k+1 \mbox{ for } k=2,\ldots,s,$$ $$j\leq
w_1-t_1+1,\; i\leq v_s-t_s+1\}$$ and by
$$\tilde{\LL}=\tilde{\LL}^+\cup\{(j,i)\mid (i,j)\in\tilde{\LL}^+\}.$$
See Figure~\ref{lplus}.

\begin{figure}[h!]
\begin{center}
\input{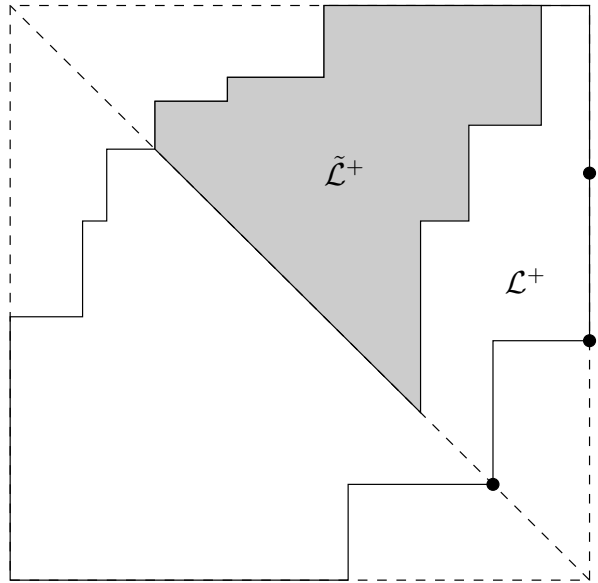}
\caption{An example of $\LL^+$ with three distinguished points and
  $t=(3,6,4)$. The corresponding $\tilde{\LL}^+$ is shaded.}
\label{lplus}
\end{center}
\end{figure}
\end{notat}

The ladder $\tilde{\LL}$ computes the height of $I_t(L)$ as follows.

\begin{prop}[\cite{go10}, Proposition~1.8]
Let $\LL$ be a ladder with distinguished points
$(v_1,w_1),\ldots,(v_s,w_s)$ and let $\tilde{\LL}$ and $\tilde{\LL}^+$
be as above. Then $\tilde{\LL}$ is a symmetric ladder and 
$$\hgt I_t(L)=|\tilde{\LL}^+|.$$
\end{prop}

The result about Gr\"obner bases will follow by combining the next
theorem with Lemma~\ref{inid}. 

\begin{thm}[\cite{go10}, Theorem~2.4]\label{biliaisonsym}
Any symmetric mixed ladder determinantal ideal can be obtained from an ideal
generated by indeterminates by a finite sequence of ascending
elementary G-biliaisons.
\end{thm}

The following is the main result of this section. We prove that the
natural generators of symmetric mixed ladder determinantal ideals are
a Gr\"obner basis with respect to any diagonal term-order, and that
the simplicial complexes associated to their initial ideals are vertex
decomposable. In particular, the initial ideals are Cohen-Macaulay.

\begin{thm}\label{gbsym}
Let $X=(x_{ij})$ be an $n\times n$ symmetric matrix of
indeterminates. Let $\LL^+=\LL^+_1\cup\ldots\cup\LL^+_s$ and
$t=(t_1,\ldots,t_s)$. Let $I_t(L)\subset K[X]$ be the
corresponding symmetric mixed ladder determinantal ideal and let
$\G_t(L)$ be the set of minors that generate it. Let $\sigma$ be any
diagonal term-order. Then $\G_t(L)$ is a reduced Gr\"obner basis of
$I_t(L)$ with respect to $\sigma$. Moreover, the initial ideal of
$I_t(L)$ with respect to $\sigma$ is squarefree and Cohen-Macaulay,
and the associated simplicial complex is vertex decomposable.
\end{thm}

\begin{proof}
Let $$I_t(L)=I_{t_1}(L_1)+\cdots+I_{t_s}(L_s)\subset K[X]$$ be
the symmetric mixed ladder determinantal ideal with ladder $\LL$,
$t=(t_1,\ldots,t_s)$ and distinguished points
$(v_1,w_1),\ldots,(v_s,w_s)$. Let $\G_t(L)$ be the set of natural
generators of $I_t(L)$. We proceed by induction on
$\ell=|\LL^+|$. 

If $\ell=1$, then $\G_t(L)$ consists of one indeterminate, in
particular it is a Gr\"obner basis of the ideal that it generates with
respect to any term-order $\sigma$. Moreover,
$in(I_1(L))=I_1(L)$ is generated by indeterminates, and the simplicial
complex associated to it is the empty set.

We now assume that the thesis holds for ladders $\N$ with 
$|\N^+|<\ell$ and prove it for a ladder $\LL$ with
$|\LL^+|=\ell$. If $t_1=\ldots=t_s=1$, then $\G_t(L)$ consists only of
indeterminates. In particular it is a Gr\"obner basis of the ideal
that it generates with respect to any term-order $\sigma$. Moreover,
$in(I_1(L))=I_1(L)$ is generated by indeterminates, and the simplicial
complex associated to it is the empty set. Otherwise, let
$k\in\{1,\ldots,s\}$ such that $t_k=\max\{t_1,\ldots,t_s\}\geq 2$. Let 
$\LL'$ be the ladder with distinguished points
$$(v_1,w_1),\ldots,(v_{k-1},w_{k-1}),(v_k+1,w_k-1), 
(v_{k+1},w_{k+1}),\ldots,(v_s,w_s)$$
and let $t'=(t_1,\ldots,t_{k-1},t_k-1,t_{k+1},\ldots,t_s)$. 
Let $I_{t'}(L')\subset K[X]$ be the associated symmetric mixed ladder
determinantal ideal. Let $\G_{t'}(L')$ be the set of minors which minimally
generate $I_{t'}(L')$. Since $|\LL'^+|<\ell$, by
induction hypothesis $\G_{t'}(L')$ is a reduced Gr\"obner
basis of $I_{t'}(L')$ with respect to
$\sigma$. Hence $$in(I_{t'}(L'))=(in(\G_{t'}(L'))).$$ 
Let $\mathcal M$ be the ladder obtained from $\LL$ by removing 
$(v_k,w_k)$ and $(w_k,v_k)$. Let
$$(v_1,w_1),\ldots,(v_{k-1},w_{k-1}),(v_k,w_k-1),(v_k+1,w_k), 
(v_{k+1},w_{k+1}),\ldots,(v_s,w_s)$$  be the distinguished points of
$\mathcal M$ and let $u=(t_1,\dots,t_{k-1},t_k,t_k,t_{k+1},\dots,t_s)$. 
Let $I_u(M)\subset K[X]$ be the associated symmetric mixed ladder
determinantal ideal. Let $\G_u(M)$ be the set of minors which minimally
generate $I_u(M)$. Since $|\M^+|=\ell-1<\ell$, by induction hypothesis
$\G_u(M)$ is a reduced Gr\"obner basis of $I_u(M)$ with respect to any
diagonal term-order. Hence $$in(I_u(M))=(in(\G_u(M))).$$

It follows from~\cite{go10}, Theorem~2.4 that $I_t(L)$ is
obtained from $I_{t'}(L')$ via an ascending elementary G-biliaison of
height $1$ on $I_u(M)$. The ideals $in(I_{t'}(L'))$ and $in(I_u(M))$
are squarefree and Cohen-Macaulay by induction hypothesis. Moreover
$$in(\G_t(L))=in(\G_u(M))\cup x_{v_k,w_k}in(\G_{t'}(L')),$$
where $x_{u,v}\G$ denotes the set of products $x_{u,v}g$ for $g\in\G$. Since
$x_{v_k,w_k}$ does not appear in $in(\G_u(M))$, it does not divide zero
modulo the ideal $in(I_u(M))$. Therefore, 
\begin{equation}\label{cpx_symm}
I:=(in(\G_t(L)))=in(I_u(M))+x_{v_k,w_k}in(I_{t'}(L'))\subseteq
in(I_t(L))
\end{equation}
and $I$ is a Basic Double G-Link of degree 1 of $in(I_{t'}(L'))$ on
$in(I_u(M))$. Hence $I$ is a squarefree Cohen-Macaulay ideal. By
Lemma~\ref{inid} $I=in(I_t(L))$, hence $\G_t(L)$ is a Gr\"obner
basis of $I_t(L)$ with respect to any diagonal term-order.

Let $\Delta$ be the simplicial complex associated to
$in(I_t(L))$. By (\ref{cpx_symm}) the simplicial complexes
associated to $in(I_{t'}(L')$ and $in(I_{u}(M))$ are
$\lk_{(v_k,w_k)}(\Delta)$ and $\Delta-(v_k,w_k)$,
respectively. $\Delta$ is vertex decomposable, since
$\lk_{(v_k,w_k)}(\Delta)$ and $\Delta-(v_k,w_k)$ are by induction
hypothesis. 
\end{proof}

\begin{rmks}
\begin{enumerate}
\item From the proof of the previous theorem it also follows that
  $in_t(L)$ is obtained from an ideal generated by indeterminates via
  a sequence of degree 1 Basic Double G-links which only involve
  squarefree monomial ideals. In particular, it is
  glicci. Moreover, the associated simplicial complex is shellable.
\item The proof of Theorem~\ref{gbsym} given above does not constitute
  a new proof of the fact that the $t$-minors in a symmetric matrix or
  in a symmetric ladder are a Gr\"obner basis with respect to any
  diagonal term-order for the ideal that they generate. In fact, our
  proof is based on Theorem~2.4 in~\cite{go10}, which in turn relies
  on the fact that minors all of the same size in a ladder of a
  symmetric matrix generate a prime ideal. Primality of this ideal is
  classically deduced from the fact that the minors are a Gr\"obner
  basis. So we are extending (and not providing a new proof of) the
  results in~\cite{co94c}.
\item Our argument, however, gives a new proof of the fact that the
  minors generating a cogenerated ideal in a symmetric matrix or in a
  ladder thereof are a Gr\"obner basis with respect to a diagonal
  term-order, knowing that minors all of the same size in a ladder of
  a symmetric matrix are a Gr\"obner basis of the ideal that they
  generate.
\end{enumerate}
\end{rmks}

\section{Mixed ladder determinantal ideals}\label{ladd_sect}

In this section, we prove that minors of mixed size in one-sided
ladders are Gr\"obner bases for the ideals that they generate, with
respect to any anti-diagonal term order. Moreover, the associated
simplicial complex is vertex decomposable. These results are already
known, and were established in different levels of generality
in~\cite{na86}, \cite{co93}, \cite{co95}, \cite{go00}, \cite{kn05},
\cite{go07b}, \cite{kn09} and~\cite{kn09p}. 
The papers~\cite{kn05}, \cite{kn09} and~\cite{kn09p} follow a
different approach than the others. The family that they treat
strictly contains that of one-sided mixed ladder determinantal
ideals. The paper~\cite{go07b} follows essentially the same approach
as the the first four papers, extending it to the family of two-sided
mixed ladder determinantal ideals. 
The proof we give here is different and independent of all the
previous ones: we use the result that we established in
Section~\ref{mainlemma} and the liaison results which were established
in~\cite{go07b}.

In~\cite{go07b} the first author approached the
study of ladder determinantal ideals from the opposite point of view:
she first proved that the minors were Gr\"obner bases of the ideals
that they generated. From the Gr\"obner basis result, she deduced that
the ideals are prime and Cohen-Macaulay, and computed their height. Finally,
she proved the liaison result. Here we  wish to take the opposite
approach: namely, deduce the fact that the minors are a Gr\"obner
basis for the ideal that they generate from the liaison result. In
order to do that, we need to show how to obtain the liaison result
independently of the computation of a Gr\"obner basis. We do this in
Appendix~\ref{ladd_app} following the approach of~\cite{de09}
and~\cite{go10}. In this section, we deduce the result about Gr\"obner
bases from the G-biliaison result.

We start by introducing the relevant notation. Let $X=(x_{ij})$ be an
$m\times n$ matrix whose entries are distinct indeterminates, $m\leq n$.

\begin{defn}\label{ladd}
A {\bf one-sided ladder} $\LL$ of $X$ is a subset of the set
$\X=\{(i,j)\in\NN^2 \mid 1\le i\le m,\ 1\le j\le n \}$
with the properties: 
\begin{enumerate}
\item $(1,m)\in\LL$,
\item if $i<h,j>k$ and $(i,j),(h,k)\in\LL$, then
$(i,k),(i,h),(h,j),(j,k)\in\LL$.
\end{enumerate}
\end{defn}

We do not make any connectedness assumption on the ladder $\LL$.
For ease of notation, we also do not assume that $X$ is the 
smallest matrix having $\LL$ as a ladder. 
Observe that $\LL$ can be written as
$$\LL=\bigcup_{k=1}^u \{(i,j)\in\X\mid i\leq c_k \mbox{ and } j\geq
d_k\}$$ for some  integers $1\leq c_1<\ldots<c_u\leq m$, $1\leq
d_1<\ldots<d_u\leq n$. 

We call $(c_1,d_1),\ldots,(c_u,d_u)$ {\bf lower outside
corners} and $(c_1,d_2),\ldots,(c_{u-1},d_u)$ {\bf lower inside
corners} of the ladder $\LL$. A one-sided ladder has at least one
lower outside corner. A one-sided ladder which has exactly one lower
outside corner is a matrix. All the corners belong to $\LL$, and the
ladder $\LL$ corresponds to its set of lower outside (or equivalently,
lower inside) corners. The {\bf lower border} of $\LL$ consists of the
elements $(c,d)$ of $\LL$ such that $(c+1,d-1)\not\in\LL$. See 
Figure~\ref{lb}. Notice that the lower corners of a ladder belong to
its lower border.  
\begin{figure}[h!]
\begin{center}
\input{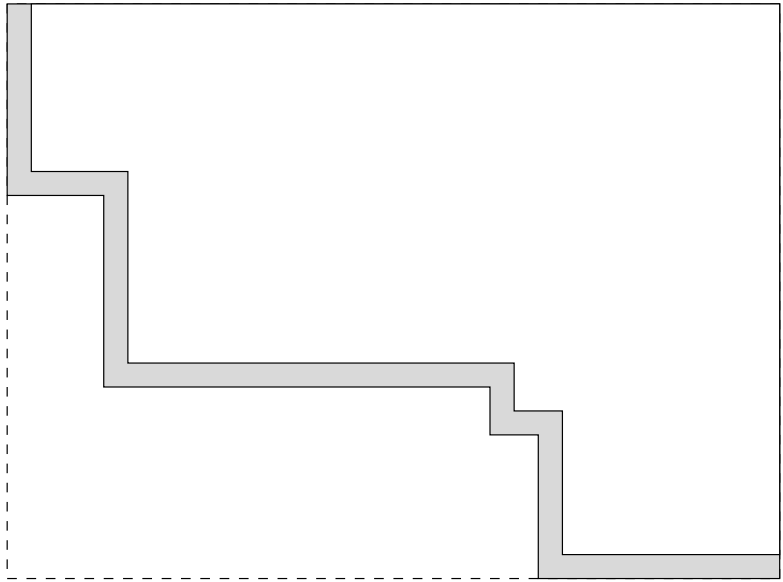}
\caption{An example of a ladder with shaded lower border.}
\label{lb}
\end{center}
\end{figure}

Given a ladder $\LL$ we set $L=\{x_{ij}\in X\mid
(i,j)\in\LL\}$. We denote by $|\LL|$ the cardinality of the ladder.
We let $I_t(L)$ denote the ideal generated by the set of the
$t$-minors of $X$ which involve only indeterminates of
$L$. In particular, $I_t(X)$ is the ideal of $K[X]$ generated by 
the $t\times t$-minors of $X$.

\begin{defn}
Let $\{(a_1,b_1),\ldots,(a_s,b_s)\}$ be a subset of the
lower border of $\LL$ which contains all the lower outside
corners. We order them so that $1\leq a_1\leq\ldots\leq a_s\leq m$
and $1\leq b_1\leq\ldots\leq b_s\leq n$.
Let $t=(t_1,\ldots,t_s)$ be a vector of positive integers. For
$k=1,\ldots,s$, denote by 
$$\LL_k=\{ (i,j)\in\X\mid i\leq a_k \mbox{ and } j\geq
b_k\}\;\; \mbox{ and }\;\; L_k=\{x_{i,j}\mid (i,j)\in\LL_k\}.$$ Notice
that $\LL_k\subseteq\LL$ and $L_k\subseteq
L$. Moreover, $\LL=\cup_{k=1}^s\LL_k.$ The ideal 
$$I_t(L)=I_{t_1}(L_1)+\ldots+I_{t_s}(L_s)$$
is a {\bf mixed ladder determinantal ideal}. 
We denote $I_{(t,\ldots,t)}(L)$ by $I_t(L)$.
We call $(a_1,b_1),\ldots,(a_s,b_s)$ {\bf distinguished points} of $\LL$.
Notice that a ladder is uniquely determined by the set of its
distinguished points, but it does not determine them.
\end{defn}

\begin{notat}\label{gensladd}
Denote by $\G_{t_k}(L_k)$ the set of the $t_k$-minors of $X$
which involve only indeterminates of $L_k$ and
let $$\G_t(L)=\G_{t_1}(L_1)\cup\ldots\cup\G_{t_s}(L_s).$$ 
The elements of $\G_t(L)$ are a minimal system of generators of
$I_t(L)$. We sometimes refer to them as ``natural generators''.
\end{notat}

We will need the following result. See the appendix for a self
contained proof.

\begin{thm}[\cite{go07b}, Theorem~2.1]\label{biliaisonladd}
Any mixed ladder determinantal ideal can be obtained from an ideal
generated by indeterminates by a finite sequence of ascending
elementary G-biliaisons.
\end{thm}

We now prove that the natural generators of a mixed ladder
determinantal ideal are a Gr\"obner basis with respect to any anti-diagonal
term-order. 

\begin{thm}\label{gbladd}
Let $X=(x_{ij})$ be an $m\times n$ matrix whose entries are distinct
indeterminates, $m\leq n$, and let $\LL$ be a one-sided ladder of
$X$. Let $t\in\NN^s$ and let $(a_1,b_1),\ldots,(a_s,b_s)$ be the
distinguished points of the ladder. Let $I_t(L)$ be the corresponding
ladder determinantal ideal. Denote by $\G_t(L)$ be the set of minors
that generate $I_t(L)$. Then $\G_t(L)$ is a reduced Gr\"obner basis of
$I_t(L)$ with respect to any anti-diagonal term ordering. Moreover,
the initial ideal $in(I_t(L))$ is squarefree Cohen-Macaulay, and the
associated simplicial complex is vertex decomposable.
\end{thm}

\begin{proof}
We proceed by induction on $\ell=|\LL|$. If $\ell=1$, then $\G_1(L)$
consists of one indeterminate. Hence $\G_1(L)$ is a reduced Gr\"obner
basis for $I_1(L)$ with respect to any term-order. Moreover,
$I_1(L)=in(I_1(L))$ is generated by indeterminates, hence the
associated simplicial complex is the empty set.

We now assume by induction that the statement holds for ladders $\HH$
with $|\HH|<\ell$, and we prove the statement for a ladder
$\LL$ with $|\LL|=\ell$. If $t=(1,\ldots,1)$, then $\G_1(L)$
consists of indeterminates. Hence $\G_1(L)$ is a reduced Gr\"obner
basis for $I_1(L)$ 
with respect to any term-order. Moreover, $I_1(L)=in(I_1(L))$ is
generated by indeterminates, hence the associated simplicial complex
is the empty set. Otherwise, let $C\subseteq in(I_t(L))$ be the ideal
generated by the initial terms of 
$\G_t(L)$. It suffices to show that $C=in(I_t(L))$ and that $C$ is
Cohen-Macaulay. Let $(a_1,b_1),\ldots,(a_s,b_s)$ be the distinguished
points of the ladder and choose $k\in\{1,\ldots,s\}$ so that
$t_k=\max\{t_1,\ldots,t_s\}\geq 2$. Let
$\M=\LL\setminus\{(a_k,b_k)\}$, and let   
$$(a_1,b_1),\ldots,(a_{k-1},b_{k-1}),(a_k-1,b_k),(a_k,b_k+1),
(a_{k+1},b_{k+1}),\ldots,(a_s,b_s)$$ be the distinguished points of
$\M$. Let $p=(t_1,\ldots,t_{k-1},t_k,t_k,t_{k+1},\ldots,t_s)\in
\NN^{s+1}$. Let $\N$ be the ladder with distinguished points 
$$(a_1,b_1),\ldots,(a_{k-1},b_{k-1}),(a_k-1,b_k+1),(a_{k+1},b_{k+1}),
\ldots,(a_s,b_s)$$ and let $q=(t_1,\ldots,t_{k-1},t_k-1,
t_{k+1},\ldots,t_s)\in\NN^s$.
As shown in the proof of Theorem~\ref{biliaisonladd}, $I_t(L)$ is obtained
from $I_q(N)$ via an elementary G-biliaison of height 1 on
$I_p(M)$. 

Let $A=in(\G_p(M))$ and let $B=in(\G_q(N))$.
The induction hypothesis applies to both $I_p(M)$ and $I_q(N)$, hence
$\G_p(M)$ is a Gr\"obner basis for $I_p(M)$, $\G_q(N)$ is a Gr\"obner
basis for $I_q(N)$ and $A=in(I_p(M)),B=in(I_q(N))$ are Cohen-Macaulay
ideals. Notice that 
$$in(\G_t(L))=in(\G_p(M))\cup x_{a_k,b_k}in(\G_q(N))$$ where $x_{a_k,b_k}\G$
denotes the set of products $x_{a_k,b_k}g$ for $g\in\G$. Since
$x_{a_k,b_k}$ does not appear in $in(\G_p(M))$, it does not divide zero
modulo the ideal $A=in(I_p(M))$. Therefore, 
\begin{equation}\label{cpx_ladd}
in(I_p(M))+x_{a_k,b_k}in(I_q(N))=C\subseteq in(I_t(L))
\end{equation} 
and $C$ is a Basic Double G-Link of degree 1 of $in(I_q(N))$ on
$in(I_p(M))$. $C$ is Cohen-Macaulay and squarefree, since $A$ and $B$
are. By Lemma~\ref{inid} we 
conclude that $C\subseteq in(I_t(L))$ and that $\G_t(L)$ is a
Gr\"obner basis of $I_t(L)$ with respect to any anti-diagonal
term-order. 

Let $\Delta$ be the simplicial complex associated to
$in(I_{t}(L))$. By (\ref{cpx_ladd}) the simplicial complexes
associated to $in(I_{q}(N)$ and $in(I_{p}(M))$ are
$\lk_{(a_k,b_k)}(\Delta)$ and $\Delta-(a_k,b_k)$,
respectively. $\Delta$ is vertex decomposable, since
$\lk_{(a_k,b_k)}(\Delta)$ and $\Delta-(a_k,b_k)$ are by induction
hypothesis. 
\end{proof}

\begin{rmks}
\begin{enumerate}
\item From the proof of the theorem it also follows that the ideal
  $in(I_t(L))$ is obtained from an ideal generated by indeterminates
  via a sequence of degree 1 Basic Double G-links, which only involve
  squarefree monomial ideals. Hence in particular it is glicci. Moreover, the
  associated simplicial complex is shellable.
\item Notice that, in contrast to Theorem~\ref{gbpf} and
Theorem~\ref{gbsym}, Theorem~\ref{gbladd} does constitute a new proof
of the fact that $t$-minors in a generic matrix or in a one-sided
ladder are a Gr\"obner basis with respect to any anti-diagonal term-order
for the ideal that they generate. In fact, in Theorem~\ref{laddall} we
give a proof of primality for mixed ladder determinantal ideals which
is independent of any previous Gr\"obner basis results.
\end{enumerate}
\end{rmks}

By following the same approach as in the previous sections and using the
result of Narasimhan from~\cite{na86}, we can prove that the natural
generators of mixed ladder determinantal ideals from two-sided ladders
are a Gr\"obner basis for the ideal that they generate with respect to
any anti-diagonal term-order (see~\cite{go07b} for the relevant
definitions). This is, to our knowledge, the largest 
family of ideals generated by minors in a ladder for which the minors
are a Gr\"obner basis for the ideal that they generate. Notice, e.g.,
that cogenerated ladder determinantal ideals all belong to this
family. The result was already established by the first author
in~\cite{go07b}, but a different proof can be given using the
techniques discussed in this paper. Notice moreover that we also show
that the simplicial complex associated to the initial ideal is vertex
decomposable. In particular, it is shellable. Since the proof is
completely analogous to the previous ones, we omit it.

\begin{thm}
Let $X=(x_{ij})$ be an $m\times n$ matrix whose entries are distinct
indeterminates, $m\leq n$, and let $\LL$ be a ladder of
$X$. Let $t\in\NN^s$ and let $(a_1,b_1),\ldots,(a_s,b_s)$ be the
distinguished points of the ladder. Let $I_t(L)$ be the corresponding
ladder determinantal ideal. Denote by $\G_t(L)$ be the set of minors
that generate $I_t(L)$. Then $\G_t(L)$ is a reduced Gr\"obner basis of
$I_t(L)$ with respect to any anti-diagonal term ordering, and the
initial ideal $in(I_t(L))$ is squarefree Cohen-Macaulay. Moreover, the
associated simplicial complex is vertex decomposable.
\end{thm}

\appendix

\section{G-biliaison of mixed ladder determinantal
  ideals from one-sided ladders}\label{ladd_app}

Mixed ladder determinantal ideals were introduced and studied by the
first author in~\cite{go07b}. They are a natural family to study, from
the point of view of liaison theory.  
In~\cite{go07b}, the first author proved
that the minors were Gr\"obner bases of the ideals that they
generated. From the Gr\"obner basis result, she deduced that the
ideals are prime, and Cohen-Macaulay, and computed their height. Finally, 
she proved the liaison result. In this article we wish to take the
opposite approach: namely, deduce the fact that the minors are a Gr\"obner
basis for the ideal that they generate from the liaison result. In
order to do that, we need to show how to obtain the liaison result
independently of the computation of a Gr\"obner basis. We do this by
showing that the approach of~\cite{de09} (for ladder ideals of
pfaffians) and~\cite{go10} (for mixed symmetric determinantal ideals)
applies also to mixed ladder determinantal ideals.
More precisely, we prove that mixed ladder determinantal ideals from
one-sided ladders are prime and Cohen-Macaulay. We also give a
different proof of the formula for their height which was given
in~\cite{go07b}. We do all these without relying on, and independently
of, the computation of a Gr\"obner basis. The liaison results
of~\cite{go07b} then follow, in particular we obtain a proof of
Theorem~\ref{biliaisonladd} which does not rely on any Gr\"obner basis
results. 

We follow the definitions and notations of Section~\ref{ladd}. The
following easily follow from the definition of a mixed ladder
determinantal ideal. 

\begin{rmks}\label{laddconv}[\cite{go07b}, Assumption~3 and Lemma~1.13]
\begin{enumerate}
\item We may assume without loss of generality that $t_k\leq\min\{a_k,
  m-b_k+1\}$ for $1\leq k\leq s.$ 
\item We may assume that $b_{k-1}-b_k<t_k-t_{k-1}<a_k-a_{k-1}$ for
  $k\geq 2$. 
\end{enumerate}
\end{rmks}

\begin{notat}\label{ltilde_notat}
Let $\LL$ be a one-sided ladder with distinguished points
$(a_1,b_1),\ldots,(a_s,b_s)$ and $t=(t_1,\ldots,t_s)$. We denote by
$\tilde{\LL}$ the one-sided ladder with lower outside corners 
$$(a_1-t_1+1,b_1+t_1-1),\ldots,(a_s-t_s+1,b_s+t_s-1).$$
$\tilde{\LL}$ is a ladder by Remarks~\ref{laddconv} (2). See
Figure~\ref{ltilde}. 
\begin{figure}[h!]
\begin{center}
\input{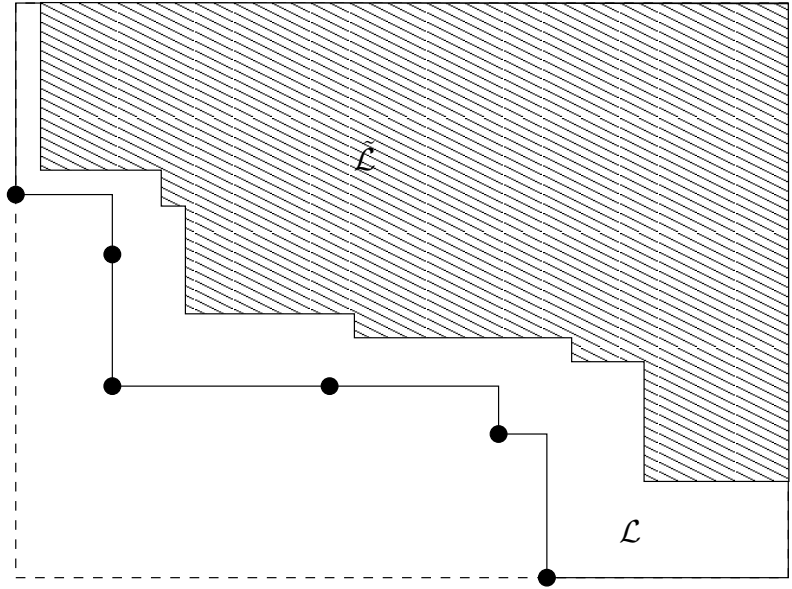}
\caption{A ladder $\LL$ with marked distinguished points. The
  subladder $\tilde{\LL}$ is shaded.}
\label{ltilde}
\end{center}
\end{figure}
\end{notat}

The height of ladder determinantal ideals from one-sided ladders was
first computed by Gonciulea and Miller in~\cite{go00},
Theorem~4.6.3. In Theorem~1.15 of~\cite{go07b}, the first author gave a new
formula for the height of ladder determinantal ideals from ladders
which are not necessarily one-sided. More precisely, she proved that
\begin{equation}\label{ltilde_eqn}
\hgt I_t(L)=|\tilde{\LL}|.
\end{equation}
Both the proofs in ~\cite{go00} and~\cite{go07b} relied on
the computation of a Gr\"obner basis of the ladder determinantal
ideal. In Theorem~\ref{laddall} we give a different proof of the height
formula (\ref{ltilde_eqn}), which is independent of the computation of
a Gr\"obner basis. 

\begin{notat}\label{local}
Let $X=(x_{ij})$ be an $m\times n$ matrix whose entries are distinct
indeterminates, $m\leq n$, and let $\LL=\LL_1\cup\ldots\cup\LL_s$ be a
one-sided ladder of $X$ with distinguished points
$(a_1,b_1),\ldots,(a_s,b_s)$. Let $(u,v)\in\LL$ and assume that
$(u,v)\in\LL_i$ for $j\leq i\leq k$ only. We let $\hat{\LL}$
denote the ladder obtained as follows: remove the entries in row $u$
and column $v$ which do not belong to any $\LL_i$ for
$i\not\in\{j,\ldots,k\}$. The remaining entries in row $u$ and column
$v$ are $(1,v),\ldots,(a_{j-1},v)\in\LL_{j-1}$ and
$(u,b_{k+1}),\ldots,(u,n)\in\LL_{k+1}$. Move the remaining entries in
row $u$ just below region $\LL_k$, i.e., between row $a_k$ and row
$a_k+1$. Move the remaining entries in column $v$ just on the left of
region $\LL_j$, i.e., between column $b_j-1$ and column $b_j$. It is
easy to check that $\hat{\LL}$ is a ladder. Rename the entries of
the ladder as needed, so that $(a,b)$ denotes the entry in position
$(a,b)$. Finally, let
$$(a_1,b_1),\ldots,(a_{j-1},b_{j-1}),
(a_j-1,b_j+1),\ldots,(a_k-1,b_k+1),
(a_{k+1},b_{k+1}),\ldots,(a_s,b_s)$$ be the distinguished points of
$\hat{\LL}$. See also Figure~\ref{lhat}.
\begin{figure}[h!]
\begin{center}
\input{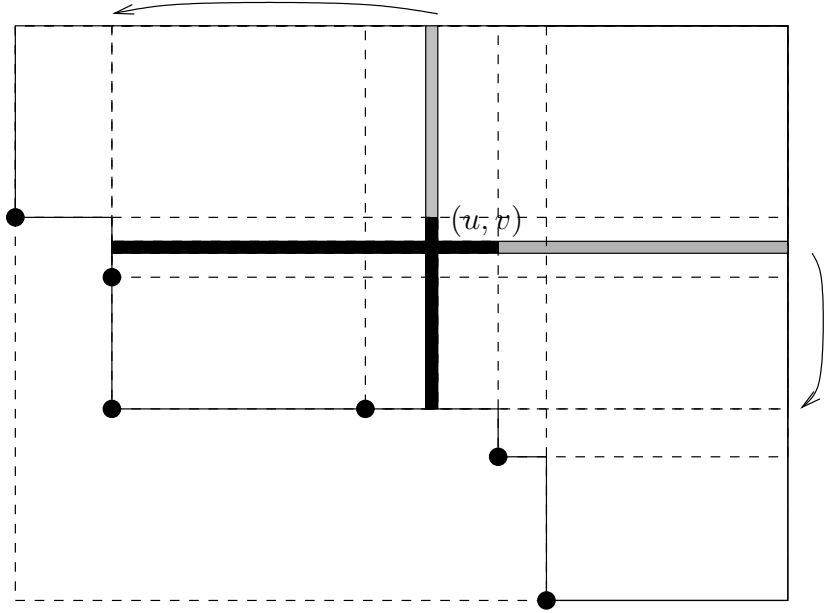}
\caption{An example of the construction of $\hat{\LL}$. The entries in
  black are deleted, while the ones in gray are moved as indicated by
  the arrows.}
\label{lhat}
\end{center}
\end{figure}
\end{notat}

We now prove a technical lemma which will be needed in the proof of
primality. The statement is essentially contained in~\cite{go00}, but
here we prove it for a larger family. The technique of the proof is a
standard one, and can be found, e.g., in~\cite{br88}.

\begin{lemma}\label{laddlocal}
Let $X=(x_{ij})$ be an $m\times n$ matrix whose entries are distinct
indeterminates, $m\leq n$, and let $\LL$ be a one-sided ladder of
$X$. Let $t\in\NN^s$ and let $$(a_1,b_1),\ldots,(a_s,b_s)$$ be the
distinguished points of the ladder. Let $I_t(L)$ be the corresponding
ladder determinantal ideal. Let $(u,v)\in\LL$ be a point of the
ladder. Let $\hat{\LL}$ be the ladder obtained as in
Notation~\ref{local}, with distinguished
points $$(a_1,b_1),\ldots,(a_{j-1},b_{j-1}), 
(a_j-1,b_j+1),\ldots,(a_k-1,b_k+1),
(a_{k+1},b_{k+1}),\ldots,(a_s,b_s).$$ Assume that $t_i\geq 2$ for
$j\leq i\leq k$ and let
$$r=(t_1,\ldots,t_{j-1},t_j-1,\ldots,t_k-1,t_{k+1},\ldots,t_s).$$
Then there is an isomorphism $$K[L]/I_t(L)[x_{u,v}^{-1}]\cong
K[\hat{L}]/I_r(\hat{L})[x_{a_{j-1}+1,v},\ldots,x_{a_k,v},
x_{u,b_j},\ldots,x_{u,b_{k+1}-1},x_{u,v}^{-1}].$$  
\end{lemma}

\begin{proof}
Under our assumptions, $\hat{\LL}$ is a ladder and $I_r(\hat{L})$
is a mixed 
ladder determinantal ideal. Let $A=K[L][x_{u,v}^{-1}]$ and
$$B=K[\hat{L}][x_{a_{j-1}+1,v},\ldots,x_{a_k,v},
x_{u,b_j},\ldots,x_{u,b_{k+1}-1},x_{u,v}^{-1}].$$
Define a $K$-algebra homomorphism
$$\begin{array}{lcl}\varphi:A & \longrightarrow & B \\
x_{i,j} & \longmapsto & \left\{\begin{array}{ll} 
x_{i,j}+x_{i,v}x_{u,j}x_{u,v}^{-1} & \mbox{if $i\neq u,
  j\neq v$ and $(i,j)\in\LL_j\cup\ldots\cup\LL_k$,} \\
x_{i,j} & \mbox{otherwise.}
\end{array}\right.\end{array}$$
The inverse of $\varphi$ is 
$$\begin{array}{lcl}\psi:B & \longrightarrow & A \\
x_{i,j} & \longmapsto &\left\{\begin{array}{ll} 
x_{i,j}-x_{i,v}x_{u,j}x_{u,v}^{-1} & \mbox{if $i\neq u,
  j\neq v$ and $(i,j)\in\hat{\LL}_j\cup\ldots\cup\hat{\LL}_k$,} \\
x_{i,j} & \mbox{otherwise.}
\end{array}\right.\end{array}$$
It is easy to check that $\varphi$ and $\psi$ are inverse to each
other. Since
$$\varphi(I_{t_i}(L_i)A)=I_{t_i-1}(\hat{L}_i)B$$ for $j\leq i\leq k$, we
have $$\varphi(I_t(L)A)=I_r(\hat{L})B\;\;\; \mbox{hence}\;\;\;
A/I_t(L)A\cong B/I_r(\hat{L})B.$$
\end{proof}

We now prove that mixed ladder determinantal ideals are prime and 
Cohen-Macaulay. We also give a proof of the height formula
(\ref{ltilde_eqn}). For the sake of clarity, we repeat the proof that
mixed ladder 
determinantal ideals are obtained from a linear space by a finite
sequence of ascending elementary G-biliaisons. The proof of primality
is adapted from the proof of Bruns and Vetter for ideals of maximal
minors of a matrix (\cite{br88}, Theorem~2.10). 

\begin{thm}[\cite{go07b}; Theorem~1.15, Theorem~1.18, Theorem~1.21,
  Theorem~2.1]\label{laddall} 
Let $\LL$ be a one-sided ladder of $X$. Let
$(a_1,b_1),\ldots,(a_s,b_s)$ be its distinguished points and let
$t\in\NN^s$. Let $I_t(L)\subset K[X]$ be the corresponding ladder
determinantal ideal. Then:
\begin{enumerate}
\item $I_t(L)$ is prime and Cohen-Macaulay.
\item Let $\tilde{\LL}$ be the one-sided ladder of
  Notation~\ref{ltilde_notat}, i.e., the ladder with lower outside corners 
$(a_1-t_1+1,b_1+t_1-1),\ldots,(a_s-t_s+1,b_s+t_s-1).$ Then
$I_t(L)\subset K[L]$ has height $$\hgt(I_t(L))=|\tilde{\LL}|.$$ 
\item $I_t(L)$ can be obtained from an ideal generated by
  indeterminates by a finite sequence of ascending elementary
  G-biliaisons. 
\end{enumerate}
\end{thm}

\begin{proof}
We prove the statement by induction on $\ell=|\LL|$. If $\ell=1$, then
$I_t(L)$ is generated by one indeterminate. In particular, it is prime
and Cohen-Macaulay. Moreover $\tilde{\LL}=\LL$ and $\hgt
I_t(L)=1=|\tilde{\LL}|$. 

We now prove the statement for a ladder $\LL$ with $|\LL|=\ell$. By
induction hypothesis, we may assume that the statement holds for any
ladder with fewer than $\ell$ entries. If $t=(1,\ldots,1)$, then
$I_t(L)$ is generated by $\ell$ indeterminates. In particular, it is prime
and Cohen-Macaulay. Moreover $\tilde{\LL}=\LL$ and $\hgt
I_t(L)=\ell=|\tilde{\LL}|$. 

Otherwise we have
$t_k=\max\{t_1,\ldots,t_s\}\geq 2$. By Remarks~\ref{laddconv} (2) we
have $a_k>a_{k-1}$ and $b_k<b_{k+1}$. Let 
$t'=(t_1,\ldots,t_{k-1},t_k-1,t_{k+1},\ldots,t_s)$ and let $\LL'$ be 
the ladder obtained from $\LL$ by removing the entries
$(a_{k-1}+1,b_k),\ldots,(a_k-1,b_k),(a_k,b_k),(a_k,b_k-1)\ldots,  
(a_k,b_{k+1}-1)$. Let $$(a_1,b_1),\dots,(a_{k-1},b_{k-1}),(a_k-1,b_k-1),
(a_{k+1},b_{k+1}),\dots,(a_s,b_s)$$ be the distinguished points of
$\LL'$. Notice that $\tilde{\LL}=\tilde{\LL'}$. Since $|\LL'|<\ell$,
by induction hypothesis $I_{t'}(L')$ is Cohen-Macaulay and has $\hgt
I_{t'}(L')=|\tilde{\LL'}|=\ell$. Moreover, $I_{t'}(L')$ can be
obtained from an ideal generated by indeterminates by a finite
sequence of ascending elementary G-biliaisons.  
Let $\M$ be the ladder obtained from $\LL$ by removing the entry
$(a_k,b_k)$, and let $$(a_1,b_1),\dots,(a_{k-1},b_{k-1}),(a_k-1,b_k),
(a_k,b_k+1),(a_{k+1},b_{k+1}),\dots,(a_s,b_s)$$ be the distinguished points of
$\M$. Let $u=(t_1,\ldots,t_{k-1},t_k,t_k,t_{k+1},\ldots,t_s)$ and let
$I_u(M)$ be the corresponding ladder determinantal ideal. Notice that
$\tilde{\M}$ is obtained from $\tilde{\LL}$ by removing the entry 
$(a_k-t_k+1,b_k+t_k-1)$. Since $|\M|=\ell-1<\ell$, by induction
hypothesis $I_u(M)$ is prime and Cohen-Macaulay, and $\hgt
I_u(M)=|\tilde{\M}|=\ell-1$.  
As shown in~\cite[Theorem~2.1]{go07b}, $I_t(L)$ is
obtained from $I_{t'}(L')$ by an ascending elementary G-biliaison of
height $1$ on $I_u(M)$. More precisely, let $f_L$ be a $t_k$-minor of
$L$ which involves rows and columns $a_k$ and $b_k$, and let $f_{L'}$
be the $(t_k-1)$-minor of the submatrix of $L'$ obtained from the
previous one by deleting rows and columns $a_k$ and $b_k$. Then 
$$I_t(L)/I_u(M)\cong I_{t'}(L')+I_u(M)/I_u(M)$$ 
and the isomorphism is given by multiplication by $f_L/f_{L'}$. By
induction hypothesis, $I_u(M)$ is Cohen-Macaulay and prime, hence also
generically Gorenstein. Therefore $I_t(L)$ is Cohen-Macaulay, since
$I_{t'}(L')$ is. Moreover, $I_{t'}(L')$ is obtained from an ideal of
linear forms by a sequence of ascending elementary
G-biliaisons. Therefore, the same holds for $I_t(L)$. 

In order to prove that $I_t(L)$ is prime, we may assume that
$t_j\geq 2$ for all $1\leq j\leq s$. In fact, if $t_1=1$ let
$\tau=(t_2,\ldots,t_s)$ and $\N=\LL_2\cup\ldots\cup\LL_s$. Then
$K[L]/I_t(L)\cong K[N]/I_{\tau}(N)$, hence $I_t(L)$ is prime if and
only if $I_{\tau}(N)$ is. Similarly if $t_s=1$. If $t_j=1$ for some
$1<j<s$, let $\tau=(t_1,\ldots,t_{j-1},t_j+1,t_{j+1},\ldots,t_s)$ and
let $\N$ be the ladder with distinguished points
$$(a_1,b_1),\ldots,(a_{j-1},b_{j-1}),(a_j+1,b_j-1),(a_{j+1},b_{j+1}),
\ldots,(a_s,b_s).$$ Let $u=a_j+1$ and $v=b_j-1$.
By Lemma~\ref{laddlocal} there is an isomorphism 
$$K[N]/I_{\tau}(N)[x_{u,v}^{-1}]\cong
K[L]/I_t(L)[x_{a_{j-1}+1,v},\ldots,x_{a_j,v},
x_{u,b_j},\ldots,x_{u,b_{j+1}-1},x_{u,v}^{-1}].$$
Since $x_{u,v}\not\in L$, $I_t(L)$ is prime if $I_{\tau}(N)$ is.

Hence we may assume without loss of generality that $t_j\geq 2$ for
all $1\leq j\leq s$. We now show that $I_t(L)$ is a prime ideal. Let
$(u,v)\in\LL$ and let $A=K[L]/I_t(L)[x_{u,v}^{-1}]$. By 
induction hypothesis and by Lemma~\ref{laddlocal} we have that $A$ is
an integral domain. Therefore, there is exactly one associated prime
ideal $P_{u,v}$ of $I_t(L)$ such that $x_{u,v}\not\in P_{u,v}$. If $P$
is the only minimal associated prime of $I_t(L)$, then $I_t(L)=P$ is
prime. Suppose by contradiction then that there is another associated
prime $Q$ of $I_t(L)$, $Q\neq P$. Since $A$ is an integral domain, it
must be $x_{u,v}\in Q$. Let $(i,j)\in\LL$ such that $x_{i,j}\not\in
Q$. Existence of such an $(i,j)$ follows from the observation that
$I_t(L)$ is unmixed of height $$\hgt(I_t(L))=|\tilde{\LL}|<|\LL|.$$
Repeating the argument above for $x_{i,j}$ instead of $x_{u,v}$, one
sees that $x_{i,j}\in P$. In particular, the image of $x_{i,j}$ in 
$A$ is zero. However, from Lemma~\ref{laddlocal} one can verify that
the image of $x_{i,j}$ in $A$ is different from zero. This yields a
contradiction, hence $I_t(L)$ is prime.
\end{proof}

\end{document}